\documentclass[letter,11pt]{amsart}
\usepackage{amsmath,amsthm,amssymb,amsfonts,enumerate,color,hyperref,esint, graphicx, pgf,tikz}

\usepackage{amssymb,amsfonts}
\usepackage[all,arc]{xy}
\usepackage{enumerate}
\usepackage{comment}
\usepackage{mathrsfs}

\usepackage{mathtools}
\usepackage{esint}
\usepackage{orcidlink}

\let\emptyset\varnothing

\newtheorem{thm}{Theorem}[section]
\newtheorem{cor}[thm]{Corollary}
\newtheorem{prop}[thm]{Proposition}
\newtheorem{lem}[thm]{Lemma}
\newtheorem{conj}[thm]{Conjecture}

\theoremstyle{definition}
\newtheorem{defn}[thm]{Definition}

\newtheorem{notn}[thm]{Notation}
\newtheorem{notns}[thm]{Notations}

\newtheorem{corcon}[thm]{Corollary of Conjecture}

\theoremstyle{remark}
\newtheorem{rem}[thm]{Remark}

\newcommand{\eps}{{\varepsilon}}

\newcommand{\RR}{{\mathbb R}}

\newcommand{\NN}{{\mathbb N}}

\newcommand\fS{{\mathfrak{S}}}

\def\phi{\varphi}

\def\eps{\epsilon }

\def\D{\partial }
\newcommand\adots{\mathinner{\mkern2mu\raise1pt\hbox{.}
\mkern3mu\raise4pt\hbox{.}\mkern1mu\raise7pt\hbox{.}}}

\renewcommand{\div}{{\rm div}}

\newcommand{\proj}[2][]{\textit{Proj}_{\vect{#1}}{#2}}
\newcommand{\vect}{\mathbf}

\renewenvironment{align}{
    \begin{equation}
    \begin{aligned}
	}
	{
    \end{aligned}
    \end{equation}
    \ignorespacesafterend
}

\makeatletter
\let\c@equation\c@thm
\makeatother
\numberwithin{equation}{section}

\title{Minimizing Eigenvalues of the Fractional Laplacian}

\author[A.~Zahl]{Alvis Zahl\orcidlink{0000-0001-6390-5913}}
    \address{Alvis Zahl. Department of Mathematics\\
    Rutgers University\\
110 Frelinghuysen Rd., Piscataway, NJ 08854, USA}
    \email{a.z@math.rutgers.edu}


\thanks{Alvis Zahl, Department of Mathematics, Rutgers University, Piscataway, NJ 08854-8019, USA} 
\thanks{Email address: a.z@math.rutgers.edu}

\calclayout

\begin{document}



\maketitle

\begin{abstract}
We study the minimizers of 
\begin{equation*}
\lambda_k^s(A) + |A|
\end{equation*}
where $\lambda^s_k(A)$ is the $k$-th Dirichlet eigenvalue of the fractional Laplacian on $A$. Unlike in the case of the Laplacian, the free boundary of minimizers exhibit distinct global behavior. Our main results include: the existence of minimizers, optimal H\"older regularity for the corresponding eigenfunctions, and in the case where $\lambda_k$ is simple, non-degeneracy, density estimates, separation of the free boundary, and free boundary regularity. We propose a combinatorial toy  problem related to the global configuration of such minimizers.
\end{abstract}


\section*{Acknowledgement}
The author gratefully acknowledges his PhD advisor, Dennis Kriventsov, for many enlightening conversations regarding this work. This work is partially supported by NSF Division Of Mathematical Sciences grant DMS-2247096.\\

\section{Introduction}
In this paper, we study the minimization problem
\begin{equation} \label{problem}
\{\lambda_k(A) + |A|: A\subset \RR^n\}
\end{equation}
where $A \subset \RR^n$, and $\lambda_k(A)$ is the $k-th$ eigenvalue of the fractional Laplacian, that is,
\begin{align} \label{evalueeqn}
\begin{cases}
(-\Delta)^s u = \lambda_k u\ & \hbox{in}~ A,\\
 u = 0 & \hbox{in}~\RR^n \setminus A,
\end{cases}
\end{align}
where $0<s<1$, and $(-\Delta)^s$ is defined as:
\begin{equation}
(-\Delta)^s u(x) := C_{n,s} \lim_{\delta \rightarrow 0^+}\int_{\{y \in \RR^n: |y-x|\geq \delta\}} \frac{u(x)-u(y)}{|x-y|^{n+2s}}dy.
\end{equation}

The interest in our free boundary problem  \eqref{problem} arises from a natural generalization of the problem of minimizing $k-th$ eigenvalue of $-\Delta$. However, even in this case, the problem is imperfectly understood. We refer the reader to \cite{henrotbook} for a detailed discussion. A classic breakthrough was given in \cite{BucurDalMaso}, which shows that such problems admit minimizers in the class of quasi-open sets, and is further improved in \cite{Bucur2012}, \cite{BucurDorinMazzoleniDarioPratelli}, \cite{MazzoleniPratelli}. The boundary regularity is recently developed in \cite{KriventsovLin1}, \cite{KriventsovLin2}, and \cite{CSY}. 

We note that \eqref{problem} can be considered as a vector-valued version of the classical Bernoulli type free boundary problem. The scalar version of the one-phase Bernoulli type problem is studied in \cite{altcaffarelli} and a more generalized two-phase version is studied in \cite{caffarelliI}, \cite{caffarelliIII}, \cite{caffarelliII}.

The fractional Laplacian has attracted many interests in the last years. Besides being a natural generalization of the Laplacian, such non-local operators often arise in different fields, including probability theory, mathematical physics, and applied sciences. For a comprehensive discussion of this type of non-local operator, we refer the reader to \cite{hitchhiker}, \cite{nonlocalbook}. Therefore, it is natural to study the free boundary problem in the setting of the fractional Laplacian. Such Bernoulli-type free boundary problems involving $(-\Delta)^s$ are studied in  \cite{CaffarelliRoquejoffreSire}, \cite{silvaroquejoffre}, \cite{silvasavinsire}, for the one-phase problem, and the two-phase problem is studied in \cite{MarkAllen2012}. One main estimate that we use is developed in \cite{kuusi_mingione_sire_2015}.

We point out that there are three main difficulties in studying problem \eqref{problem}. Firstly, the eigenvalues are given by a min-max of the Rayleigh quotient which is a critical point of the functional instead of the minimum in, for example, the Bernoulli problem. We overcome this difficulty by following the idea of \cite{Bucur2012}. A major problem arises when the multiplicity of the minimizing eigenvalue is not $1$, in which case, eigenfunctions might be degenerate. 

Second, $(-\Delta)^s$ is non-local, which poses more technical difficulties than the local case, where we used estimates and techniques in \cite{kuusi_mingione_sire_2015}. The non-locality also gives rise to the difficulty in studying the global configuration of minimizers due to the interaction between different pieces of the domain. To capture the behavior at infinity, our strategy is to study minimizers on a union of different copies of $\RR^n$ that don't interact with each other (Definition \ref{defmini}). 

Third, it is not a priori clear over what class of sets that we are minimizing over. To deal with this, we rephrase the problem using level sets of $W^{s,2}$ functions. We note it is not necessary to do so and one could follow the idea of \cite{Bucur2012} and consider the minimization problem \eqref{problem} in the class of appropriate quasi-open sets. The main benefit of our approach is that we avoid the potential theory involved completely, and we refer the reader to \cite{potentialbook} for discussions on capacity and quasi-open sets. 

\hfill\\
\textbf{Main Results.}

\begin{thm}\label{A}
There exists an open bounded measurable set $A \subset \cup_{i=1}^k \RR^n$ such that the minimum of \eqref{problem} is achieved, that is
\begin{equation}
\lambda(A) + |A| \leq \inf \{\lambda_k(A') + |A'|\}.
\end{equation}
\end{thm}
Here the infimum is taken over a set of appropriate domains in $\cup_{i=1}^k \RR^n$, which will be made precise in the next section. We point out that in the local case, minimizers are given by unions of disjoint sets \cite{Bucur2012}, which are translation invariant. In the nonlocal case, it could happen that these disjoint pieces are infinitely  away from each other, and thus we consider the existence of minimizers in copies of $\RR^n$.

\begin{thm}\label{B}
There exists an eigenfunction $u_k$, corresponding to the minimizing eigenvalue $\lambda_k$ of \eqref{problem}, such that $u_k \in C^{0,s}(\overline{A})$.
\end{thm}

Similar to the local case discussed in \cite{BucurDorinMazzoleniDarioPratelli}, the perturbation of eigenfunctions gives information about a linear combination of the eigenfunctions corresponding to $\lambda_k(A)$, and therefore, Theorem \ref{B} only guarantee H\"older continuity for one eigenfunction.

If $\lambda_k(A)$ is simple, we moreover have the following results.

\begin{thm}\label{C}
Suppose the minimizing eigenvalue $\lambda_k$ in \eqref{problem} is simple. Let $u_k$ be the corresponding eigenfunction to $\lambda_k$, then
\begin{enumerate}
\item (Non-degeneracy.) For all $r_0>0$, there exists $r<r_0$, $c_0$, such that for $|B_{r/8}(x_0)\cap A| >0$, we have
\begin{equation}
\sup_{B_{r}(x_0)} |u_k| \geq c_0r^s. \nonumber
\end{equation}
\item (Separation of the free boundary.)
\begin{equation*}
\D\{u_k(x) > 0\} \cap \D\{u_k(x) < 0\} = \emptyset.
\end{equation*}
\item (Density estimate.) Suppose $x$ is a free boundary point, then there exists $C, r_0$ such that for all $r<r_0$ small, we have
\begin{align} 
&|\{u_k > 0\} \cap B_r(x)| \geq Cr^n, \\
&|\{u_k = 0\} \cap B_r(x)| \geq Cr^n, \nonumber
\end{align}
and $|\D A| = 0$.
\item (Free boundary regularity.) Let $u$ be the extension of $u_k$ (defined in section \ref{optimalregularity}). Assume that $\|u\|_{C^{0,s}(B_1)} \leq C$, and $|u-U|\leq \tau_0$ in $B_1$ where $U$ (depending only on $x_n,y$), in polar coordinates $x_n = \rho \cos(\theta)$, $y = \rho \sin(\theta)$, is
\begin{equation*}
U(x,y) = \left(\rho^{1/2}\cos(\theta/2)\right)^{2s}.
\end{equation*}
Then $F(u) = \D \{u_k \neq 0\}$ is $C^{1,\gamma_0}$ in $B_{1/2}$, for some constant $\gamma_0$.
\end{enumerate}
\end{thm}
\hfill\\
\textbf{Outline of paper.}

In section \ref{Preliminaries}, we give a formal definition of \eqref{problem}. In particular, we extend our domain to a union of $k$ copies of $\RR^n$. We introduce notation, and prove some useful preliminary estimates.

Sections \ref{Section:Bounded} and \ref{Section:Exist} deal with the existence of minimizers. We first prove an a priori estimate showing that shape minimizers have to be bounded following the idea of \cite{Bucur2012}. Using the boundedness of minimizers we show the existence of shape minimizers of \eqref{problem}. 

In Section \ref{optimalregularity}, we prove optimal regularity for some eigenfunctions. However, it is not clear which eigenfunction has such regularity due to the difficulty mentioned earlier, and we assume that $\lambda_k$ is simple in the rest of the paper. In Section \ref{Section:nondegenerate}, we prove that the corerrsponding eigenfunction is non-degenerate. 

In Section \ref{sectionblowup}, we follow the idea of \cite{MarkAllen2012} and conduct blow-up analysis using a Weiss type monotonicity formula and we conclude that the free boundaries are separated in Section \ref{Section:separation}. This behavior of separation also yields density estimates. In Section \ref{freeboundary}, we show that the eigenfunction corresponding to the minimizing $\lambda_k$ is an almost minimizer in the setting of \cite{AllenGarcia}, which implies that the free boundary is $C^{1,\gamma_0}$.

In the Section \ref{Section:global}, we study the stability of the shape minimizer under translations and pose a toy problem that helps us in understanding the global configuration. In particular, we conjecture that there is no stable configurations where the disconnected pieces can interact with each other, and the minimizer has to be a union of $k$ congruent balls in different copies of $\RR^n$. In particular, this implies that $\lambda_k$ is not simple.

\section{Preliminaries}\label{Preliminaries}
We introduce notations, definitions, and preliminary results in this section. 

\begin{defn} Let $\Omega\subset \RR^n$ be any measurable set. Define the Sobolev space $\tilde{W}^{s,2}(\Omega)$ as the completion of $C^\infty(\Omega)$ with respect to the Gagliardo norm
\begin{equation}
[u]^2_{W^{s,2}}:=  \int_{\RR^n}\int_{\RR^n} \frac{|u(x)-u(y)|^2}{|x-y|^{n+2s}}dxdy.
\end{equation}
\end{defn}

We always understand the problem \eqref{problem} in the following sense.

\begin{defn}\label{defmini}
For any $k>0$, define
\begin{equation}
\mathcal{R}^n_k = \cup_{i=1}^k \RR_i^n,\nonumber
\end{equation}
Here $\RR^n_i$ is identical to $\RR^n$ indexed by $i$. We say $u \in \tilde{W}^{s,2}(\mathcal{R}^n_k)$ if $u^i = u|_{\RR^n_i} \in \tilde{W}^{s,2}(\RR^n)$. Consider the following class of functions
\begin{equation} \label{classu}
\mathfrak{S}_k = \{u\in \tilde{W}^{s,2}(\mathcal{R}^n_k): |{\{u\neq 0\}}|<C\}.
\end{equation}
Define the Rayleigh quotient 
\begin{equation}
\mathfrak{R}_k = \frac{1}{\int_{\mathcal{R}_k^n} |u|^2} \int_{\RR^n}\int_{\RR^n}\sum_{i,j = 1}^{k}\frac{(u^i(x) - u^j(y))^2}{|x-y|^{n+2s}}dxdy.\nonumber
\end{equation}
We understand problem \eqref{problem} in the following sense. 
\begin{equation}\label{Problem}
\lambda_k + |\{u \neq 0\}| = \min_{V_k \subset \mathfrak{S}_k} \max_{u \in V_k} \frac{1}{\int_{\mathcal{R}^n_k} |u|^2}  \int_{\RR^n}\int_{\RR^n}\sum_{i = 1}^{k}\frac{(u^i(x) - u^i(y))^2}{|x-y|^{n+2s}}dxdy + |\{u \neq 0\}|,
\end{equation}
where the minimum is taken over $k$ dimensional subspace $V_k$ of $\mathfrak{S}_k$. 
If $u_k$ minimizes \eqref{Problem}, we say 
\begin{equation}
A = \{u \neq 0\} = \{u^1 \neq 0\} \cup \cdots \cup \{u^k \neq 0\}\nonumber
\end{equation}
a \textit{$\lambda_k$-minimizer}, call $\lambda_k(A)$ the \textit{minimizing eigenvalue} and $u_k$ the \textit{corresponding eigenfunction}.
\end{defn}

\begin{rem}
The motivation for above definition is due to a concentration compactness phenomenon. Unlike in the local case where one can translate the domain, it could happen that a minimizing sequence of domains has disjoint pieces that move away from each other. We consider the limiting ``infinitely away pieces" as lying in different copies of $\RR^n$ and don't interact with each other. When $\mathcal{R}^n_k = \RR^n$, the above definition of $\mathfrak{R}_k$ is the usual Rayleigh quotient. \\
\end{rem}
\begin{rem}\label{remdefn2}
It is intentional that our space $\mathfrak{S}_k$ and the eigenvalue that we are interested in minimizing $\lambda_k$ depend on the same $k$, since for minimizing $\lambda_k$, we need at most $k$ copies of $\RR^n$. In particular, for any $k' > k$, any k dimensional subspace of $\mathfrak{S}_{k'}$ is contained in $\mathfrak{S}_{k}$, and thus \eqref{Problem} is equivalent to
\begin{align}
\min_{V_k \subset \mathfrak{S}_{k'}} \max_{u \in V_k} \frac{1}{\int_{\mathcal{R}^n_{k'}} |u|^2}  \int_{\RR^n}\int_{\RR^n}\sum_{i = 1}^{k'}\frac{(u^i(x) - u^i(y))^2}{|x-y|^{n+2s}}dxdy +\{u \neq 0\}. \nonumber
\end{align}
It could happen that we only need $m <k$ copies, and in this case we have $u|_{\RR^n_j} = 0$ for all $j > m$.
\end{rem}

\begin{notns}
If $x,y \in \mathcal{R}^n_k$, and $x \in \RR^n_i, y\in \RR^n_j$, with $i\neq j$, we formally define
\begin{equation}
|x-y| = +\infty\quad \text{ and }\quad \frac{1}{|x-y|} = 0.\nonumber
\end{equation}
In particular, for $u \in \tilde{W}^{s,2}(\mathcal{R}^n_k)$, we have
\begin{align}
\int_{\RR^n}\int_{\RR^n}\sum_{i,j = 1}^{k}\frac{(u^i(x) - u^j(y))^2}{|x-y|^{n+2s}}dxdy = \int_{\RR^n}\int_{\RR^n} \frac{(u(x)-u(y))^2}{|x-y|^{n+2s}}dxdy, \nonumber
\end{align}
and
\begin{equation} \label{rayleighrn}
R[u] := \mathfrak{R}_k[u] = \frac{1}{\int_{\mathcal{R}_k^n} |u|^2} \int_{\RR^n}\int_{\RR^n}\frac{(u(x) - u(y))^2}{|x-y|^{n+2s}}dxdy.
\end{equation}
For any measurable $A \subset \mathcal{R}^n_k$, \begin{equation}
\mathfrak{S}_k(A) = \{u \in \mathfrak{S}_k: u = 0 \text{ a.e. on } \mathcal{R}^n_k \setminus A\}.\nonumber
\end{equation}
Let $K(x,y)$ be the kernel: 
\begin{equation}
K(x,y) = \frac{1}{|x-y|^{n+2s}}.
\end{equation}
Let $B[\cdot,\cdot]$ be the inner product of  $\tilde{W}^{s,2}(\mathcal{R}_k^n)$
\begin{equation}
B[u,v] = \int_{\RR^n}\int_{\RR^n}\frac{(u(x)-u(y))(v(x)-v(y))}{|x-y|^{n+2s}}dxdy.\nonumber
\end{equation}
We sometimes omit the domain of integration if the domain the whole space, and we omit $dxdy$ when there is no confusion.
\end{notns}

\begin{defn}
For any measurable $\Omega\in \mathcal{R}^n_k$, we say $u$ is a weak solution to
\begin{align}
\begin{cases}
(-\Delta)^su= f\ & \hbox{in}~\Omega,\\
 u = 0 & \hbox{in}~\mathcal{R}^n_k \setminus \Omega,
\end{cases}\nonumber
\end{align}
if $u \in \mathfrak{S}_k(\Omega)$ and 
\begin{equation}
\int_{\RR^n}\int_{\RR^n} \frac{(u(x)-u(y))}{|x-y|^{n+2s}}(\phi(x)-\phi(y))dxdy = \int_\Omega u\phi dx.\nonumber
\end{equation}
for all $\phi \in C^\infty_c(\mathcal{R}^n_k)$, that is $\phi|_{\RR^n_i} \in C^\infty_c(\RR^n)$.
\end{defn}

We proceed to prove some useful estimates. 

\begin{prop} \label{ulinf}
If $u \in \tilde{W}^{s,2}(\mathcal{R}^n_k)$ satisfies in the weak sense:
\begin{equation}
\begin{cases}
(-\Delta)^su=\lambda u\ & \hbox{in}~\{u\neq 0\},\\
 u = 0 & \hbox{in}~\mathcal{R}^n_k \setminus \{u\neq 0\},\nonumber
\end{cases}
\end{equation}
Then $u \in L^\infty(\mathcal{R}^n_k)$ and 
\begin{equation}
\|u\|_{L^\infty(\mathcal{R}^n_k)} \leq C \lambda^{\frac{n}{4s}}\|u\|_{L^2(\mathcal{R}^n_k)}.\nonumber
\end{equation}
\end{prop}

\begin{proof}
Let $\Omega = \{u \neq 0\}$. Define $v(t,x) = e^{-\lambda t}u$ on $\Omega \times [0,T)$ for some $T >0$, then $v\in L^2(\tilde{W}^{s,2}_0(\Omega))$ is a weak solution to the heat equation:
\begin{equation}
\begin{cases}
\D_t v + (-\Delta)^s v = 0 & \hbox{in}~\Omega\times [0,\infty),\\
v(0,x) = u(x) & \hbox{in}~\Omega\times \{t = 0\}.
\end{cases}\nonumber
\end{equation}
For every $t >0$, we extend $v(t,x)$ to $\mathcal{R}_k^n$ by $0$ and still denote it $v$. Then we have in the weak sense:
\begin{equation} \label{2.4}
\D_t|v| + (-\Delta)^s |v| \leq 0
\end{equation}
We want to compare $v$ with another function $w$ defined in the following ways. First extend $u$ to $\mathcal{R}_k^n$ by $0$ and denote the extended function by $\tilde{u}$. Let $w$ be the solution to the following heat equation:
\begin{equation} \label{wheat}
\begin{cases}
\D_t w + (-\Delta)^s w = 0 & \hbox{in}~\mathcal{R}^n_k\times [0,\infty),\\
w(0,x) = \tilde{u}(x) & \hbox{in}~\mathcal{R}^n_k\times \{t = 0\}.
\end{cases}
\end{equation}
We claim that:
\begin{equation}
|v| \leq w.\nonumber
\end{equation}
Indeed, $|v| - w$ is a subsolution and 
\begin{equation}
\max_{\D \Omega \times [0,T)} |v| - w = \max_{\D \Omega \times [0,T)} - w \leq 0,\nonumber
\end{equation}
where the last inequality used the maximum principle for $w$. Using maximum principle again, we conclude that $|v| \leq w$. \\
We also know:
\begin{equation}
w(t,x) = \int_{\mathcal{R}^n_k}p(x-y,t)u(y)dy,\nonumber
\end{equation}
where $p(x,t)$ is the fractional heat kernel. It is known, for example from \cite{ChenKumagai}, 
\begin{align}
&p(x,t) = t^{-\frac{n}{2s}}p(1,t^{-\frac{1}{2s}}x).
\nonumber
\end{align}
Therefore,
\begin{align}
|v(t,x)| \leq w(t,x) \leq \int_{\mathcal{R}^n_k}|p(x-y,t)||u(y)|dy \leq \|p(x-y,t)\|_{L^2(\mathcal{R}^n_k)}\|u\|_{L^2(\mathcal{R}^n_k)} \leq Ct^{-\frac{n}{4s}}\|u\|_{L^2(\mathcal{R}^n_k)}.\nonumber
\end{align}
Finally, we observe that:
\begin{equation}
|u(x)| = |e^{\lambda t}v(t,x)| \leq Ce^{\lambda t}t^{-\frac{n}{4s}}\|u\|_{L^2(\mathcal{R}^n_k)} \leq C \lambda^{\frac{n}{4s}}\|u\|_{L^2(\mathcal{R}^n_k)},\nonumber
\end{equation}
where the last inequality follows from choosing $t = \frac{n}{4s\lambda}$.
\end{proof}

\begin{lem} (Cacciopoli inequality.) Let $u \in \fS_k$ be the weak solution of
\begin{equation}\label{2.29}
(-\Delta)^s u = f \text{ in } \{u \neq 0\},
\end{equation} 
where $f\in L^1(U)$, for some $U \subset \mathcal{R}^n_k$ open, containing $\{u \neq 0\}$. Then for any $B_r \subset \{u \neq 0\}$,
\begin{equation}\label{Cacciopolli}
\int_{B_{9r/10}}\int_{B_{9r/10}} \frac{(u(x)-u(y))^2}{|x-y|^{n+2s}}\leq \sup_{B_r} |u| \|f\|_{L^1(B_r)}+ Cr^{n-2s}\|u\|_{L^\infty(\mathcal{R}^n_k)}\sup_{B_r}|u|.
\end{equation}
\end{lem}
\begin{proof}
Let $\eta$ be a smooth cutoff such that $\eta = 1$ in $B_{9r/10}$ and $\eta = 0$ outside $B_r$, and $|\nabla \eta|\leq Cr^{-1}$. Multiply \eqref{2.29} by $\eta^2u$ and integration by part gives
\begin{equation}
\int_{\RR^n}\int_{\RR^n}\frac{(u(x)-u(y))(\eta^2(x)u(x)-\eta^2(y)u(y))}{|x-y|^{n+2s}} = \int_{\mathcal{R}^n_k} f\eta^2 u \leq \sup_{B_r} |u| \|f\|_{L^1(B_r)}. \nonumber
\end{equation}
We note that for the left hand side,
\begin{equation}
\int_{\RR^n}\int_{\RR^n} = \int_{B_{r}}\int_{B_{r}} + 2\int_{B_{r}}\int_{\RR^n \setminus B_{r}} \nonumber
\end{equation} 
with the same integrand. We note that the first term above dominates what we are estimating, and we proceed to estimate the second term. We note that for the integrand, we can rewrite as follows
\begin{align}
(u(x)-u(y))(\eta^2(x)u(x)-\eta^2(y)u(y)) &= (\eta(x)u(x) - \eta(y)u(y))^2 - u(x)u(y)(\eta(x)-\eta(y))^2.\nonumber
\end{align}
We note that the first term is positive, and for the second term, we have
\begin{align}
\int_{B_r}\int_{\RR^n}\frac{u(x)u(y)(\eta(x)-\eta(y))^2}{|x-y|^{n+2s}} &=\int_{B_r}\int_{|x-y|<r}\frac{u(x)u(y)(\eta(x)-\eta(y))^2}{|x-y|^{n+2s}} \\&+ \int_{B_r}\int_{|x-y|>r}\frac{u(x)u(y)(\eta(x)-\eta(y))^2}{|x-y|^{n+2s}}.\nonumber
\end{align}
For the $|x-y|<r$ piece, we use that $|\eta(x)-\eta(y)|< Cr^{-1}|x-y|$, 
\begin{align}
\int_{B_r}\int_{|x-y|<r}\frac{u(x)u(y)(\eta(x)-\eta(y))^2}{|x-y|^{n+2s}} &\leq C\|u\|_{L^\infty(\mathcal{R}^n_k)}\sup_{B_r} |u| r^{-2}\int_{B_r}\int_{|x-y|<r}\frac{|x-y|^2}{|x-y|^{n+2s}} \\&\leq Cr^{n-2s}\|u\|_{L^\infty(\mathcal{R}^n_k)}\sup_{B_r} |u|.\nonumber
\end{align}
For the $|x-y|>r$ piece, we use $|\eta| < 1$,
\begin{align}
\int_{B_r}\int_{|x-y|>r}\frac{u(x)u(y)(\eta(x)-\eta(y))^2}{|x-y|^{n+2s}} &\leq C\|u\|_{L^\infty(\mathcal{R}^n_k)}\sup_{B_r} |u|\int_{B_r}\int_{|x-y|>r}\frac{1}{|x-y|^{n+2s}} \\&\leq Cr^{n-2s}\|u\|_{L^\infty(\mathcal{R}^n_k)}\sup_{B_r}|u|.\nonumber
\end{align}
Combining all the terms, we have
\begin{equation}
\int_{B_{9r/10}}\int_{B_{9r/10}} \frac{(u(x)-u(y))^2}{|x-y|^{n+2s}}\leq \sup_{B_r} u \|f\|_{L^1(B_r)} + Cr^{n-2s}\sup_{B_r}|u|.\nonumber
\end{equation}
\end{proof}

\section{Boundedness of Shape Minimizers} \label{Section:Bounded}

In this section, we make use of the observation of Bucur in \cite{Bucur2012} that our shape minimizer, $A = \{u \neq 0\}$ in problem \eqref{Problem} is a priori a local shape subsolution. Then we prove that if $A$ is a local shape subsolution, then it must be bounded.\\

Define the torsion energy:
\begin{equation} \label{torenergy}
E(A) := \min_{u\in \tilde{W}^{s,2}(\mathcal{R}^n_k)}\frac{1}{2}\int_{\RR^n}\int_{\RR^n}\frac{(u(x)-u(y))^2}{|x-y|^{n+2s}}dxdy - \int_{\mathcal{R}^n_k} u(x) dx.
\end{equation}
We note that if $u_A$ minimizes $E(A)$, then $u_A$ satisfies
\begin{align}\label{torpde}
\begin{cases}
(-\Delta)^s u_A - 1 = 0\ & \hbox{in}~A\\
 u_A = 0 & \hbox{in}~\mathcal{R}^n_k \setminus A
\end{cases}
\end{align}
in the weak sense. It follows that $u_A \geq 0$. Define the corresponding $\gamma$ distance between $A,B$ as 
\begin{equation}
d_\gamma(A,B) : = \int_{\mathcal{R}^n_k} |u_A-u_B| dx
\end{equation}
where $u_A$, $u_B$ minimize $E(A),E(B)$ respectively and are extended by $0$.\\

We first prove a uniform bound of $E(A)$. 

\begin{defn} (Symmetric decreasing rearrangement \cite{almgrenlieb}.) Let $f:\mathcal{R}^n_k \rightarrow \RR^+$, $\alpha(n)$ denote the volume of unit ball in $\RR^n$. Define the radius $R_f(y)$ to be such that
\begin{equation}
\alpha(n)R_f(y)^n = \int_{\mathcal{R}^n_k} \chi_{\{f(x)>y\}}dx.
\end{equation}
We define the symmetrization $f^*$ of $f$ to be
\begin{equation}
f^*(x) = \sup\{y: R_f(y)>|x|\} = \int_0^\infty \chi_{\{B_0(R_f(y))\}}(x)dy.
\end{equation}
\end{defn}

\begin{lem}
\begin{equation}
E(B) \leq E(A) (\leq 0)
\end{equation}
where $B$ is a ball and $|B| = |A|$.
\end{lem}
\begin{proof}
Let $w_A$ be a function that attains the minimum of \eqref{torenergy}. Let $w_A^*$ be the symmetric rearrangement of $w_A$ defined above, then it follows from  Theorem  9.2\cite{almgrenlieb}  that
\begin{align}
&\int_{\RR^n} w_A^* dx = \int_{\mathcal{R}^n_k} w_A dx, \\
&B[w_A^*,w_A^*] \leq  B[w_A,w_A].
\end{align}
Therefore,
\begin{equation}
E(B) \leq \frac{1}{2} B[w_A^*,w_A^*]- \int_{\RR^n} w_A^* dx \leq \frac{1}{2} B[w_A,w_A] - \int_{\mathcal{R}^n_k} w_A dx = E(A).
\end{equation}
\end{proof}

\begin{cor}
\begin{equation}
E(A) \leq C|A|^{\frac{n+2s}{n}}.
\end{equation}
\end{cor}
\begin{proof}
It is known from \cite{RosOtonSerra2014} that for any $r>0$, $x_0 \in \RR^n$, the solution of 
\begin{align}
\begin{cases}
(-\Delta)^s u = 1\ & \hbox{in}~B_r(x_0)\\
 u = 0 & \hbox{in}~\RR^n \setminus B_r(x_0)
\end{cases}
\end{align}
is given explicitly by 
\begin{equation}
u(x) = \frac{2^{-2s}\Gamma(\frac{n}{2})}{\Gamma(\frac{n+2s}{2})\Gamma(1+s)}(r^2-|x-x_0|^2)^s.
\end{equation}
Therefore, for $|B| = Cr^n$, we have
\begin{equation}
|E(A)|\leq |E(B)| = |\frac{1}{2}\int_{\RR^n}u| \leq Cr^{n+2s}.
\end{equation}
Since $|A| = |B|$, the inequality of the Corollary follows.
\end{proof}

\begin{defn}\label{defn:subsolution}
We say that a measurable set $A$ of finite measure is a local shape subsolution for the torsion energy problem \eqref{torenergy} if there exists $\delta >0$ and $\Lambda >0$ such that for every measurable set $\tilde{A}$ satisfying $|A\setminus \tilde{A}|=0$, $d_\gamma(A,\tilde{A})\leq \delta$,  we have:
\begin{equation} \label{torsubsolution}
E(A) + \Lambda |A| \leq E(\tilde{A}) + \Lambda|\tilde{A}|.
\end{equation}
\end{defn}

By an observation of Bucur in \cite{Bucur2012}, the $\lambda_k$-minimizer of our problem \eqref{problem} is a local subsolution in the sense of above definition. The proof of the following theorem is given in the Appendix.

\begin{thm}\label{minissub}
Let $A$ be a $\lambda_k$-minimizer of \eqref{problem}. Then $A$ is a local shape subsolution for the torsion energy problem \eqref{torenergy}.
\end{thm}

The key ingredient in proving the boundedness of subsolutions is the following non-degeneracy lemma.

\begin{lem}
If $A$ is a local shape subsolution for the torsion energy problem, let $u$ minimizes the torsion energy $E(A)$. Then there exists $r_0 > 0$, $C_0 = C(n,s)$, such that either \begin{equation} \label{nondetor}
\sup_{B_{r_0}(0)} u \geq C_0{r_0}^{s}
\end{equation}
or $u = 0$ in $B_{{r_0}/2}(0)$.
\end{lem}
\begin{rem}
The inequality \eqref{nondetor} can be equivalently written as
\begin{equation}
\sup_{B_{r_0}(0)} u \geq C(n,s).
\end{equation} 
since $r_0$ is a fixed constant on the right hand side of \eqref{nondetor} and is written in this form for convenience of proof.
\end{rem}
\begin{proof}
Let $\eta$ be a smooth cutoff function such that $\eta = 0$ in $B_{3r/4}$, $\eta = 1$ outside $B_{9r/10}$ and $|\nabla \eta| \leq Cr^{-1}$. Let $v = \eta u$ be a competitor of \eqref{torenergy}. By \eqref{torsubsolution}, we have
\begin{align} \label{3.25}
\Lambda |\{u>0\}\cap B_{3r/4}| &\leq E(A\setminus B_{3r/4}) - E(A) \\&\hspace{-0.5 in}\leq \frac{1}{2}\int_{\RR^n}\int_{\RR^n}\frac{(v(x)-v(y))^2}{|x-y|^{n+2s}} - \frac{(u(x)-u(y))^2}{|x-y|^{n+2s}}dxdy - \int_{\mathcal{R}^n_k} (v(x)-u(x)) dx.
\end{align}
We first estimate the second term, 
\begin{equation}
\left|\int_{\mathcal{R}^n_k} (v(x)-u(x)) dx\right| = \left|\int_{B_{9r/10}} (\eta(x)-1)u(x) dx\right| \leq Cr^n\sup_{B_r} u. \nonumber
\end{equation}
For the double integral term, we can split the integral into two pieces,
\begin{equation}
\int_{\RR^n}\int_{\RR^n} = \int_{B_{9r/10}}\int_{B_{9r/10}}+2\int_{B_{9r/10}}\int_{\RR^n \setminus {B_{9r/10}}}.\nonumber
\end{equation}
On $B_{9r/10} \times B_{9r/10}$, we first note that the numerator is
\begin{align}
(\eta(x)u(x) - \eta(y)u(y))^2 &= (\eta(x)(u(x)-u(y)) + u(y)(\eta(x)-\eta(y)))^2\\&
\leq 2(\eta^2(x)(u(x)-u(y))^2 + u^2(y)(\eta(x)-\eta(y))^2). \nonumber
\end{align}
We use Cacciopoli inequality \eqref{Cacciopolli} for the integral corresponding to the first term and that $|\eta(x)-\eta(y)| \leq Cr^{-2}|x-y|$ for the second term,
\begin{align}
&\int_{B_{9r/10}}\int_{B_{9r/10}} \frac{\eta^2(x)(u(x)-u(y))^2}{|x-y|^{n+2s}} \leq Cr^{n-2s}\sup_{B_r}u ,
\\&\int_{B_{9r/10}}\int_{B_{9r/10}} \frac{u^2(y)(\eta(x)-\eta(y))^2}{|x-y|^{n+2s}} \leq C r^{-2}\sup_{B_r}u^2 \int_{B_{9r/10}}\int_{B_{9r/10}} \frac{|x-y|^2}{|x-y|^{n+2s}} \leq Cr^{n-2s}\sup_{B_r} u^2. \nonumber
\end{align}
On $B_{9r/10} \times (\RR^n \setminus B_{9r/10})$, using $\eta = 1$ on $\RR^n \setminus B_{9r/10}$, we have
\begin{align} 
(\eta(x)u(x)-\eta(y)u(y))^2 - (u(x)-u(y))^2 &= (u(x)-u(y) + u(x)(\eta(x)-\eta(y)))^2 - (u(x)-u(y))^2
\\&= u^2(x)(\eta(x)-\eta(y))^2 + 2 u(x)(\eta(x)-\eta(y))(u(x)-u(y)).\nonumber
\end{align}
For the first term above, 
\begin{align}
\int_{B_{9r/10}}\int_{\RR^n \setminus {B_{9r/10}}}\frac{u^2(x)(\eta(x)-\eta(y))^2}{|x-y|^{n+2s}} = &\int_{B_{9r/10}}\int_{|x-y|<r}\frac{u^2(x)(\eta(x)-\eta(y))^2}{|x-y|^{n+2s}} \\&+\int_{B_{9r/10}}\int_{|x-y|>r}\frac{u^2(x)(\eta(x)-\eta(y))^2}{|x-y|^{n+2s}}. \nonumber
\end{align}
Using that $|\eta(x)-\eta(y)| < Cr^{-1}|x-y|$ when $|x-y|<r$, and that $\eta \leq 1$ when $|x-y|>r$, we have
\begin{align}
\int_{B_{9r/10}}\int_{|x-y|<r}\frac{u^2(x)(\eta(x)-\eta(y))^2}{|x-y|^{n+2s}} &\leq C r^{-2}\sup_{B_r} u^2 \int_{B_{9r/10}}\int_{|x-y|<r}\frac{|x-y|^2}{|x-y|^{n+2s}} \leq Cr^{n-2s}\sup_{B_r} u^2,\\
\int_{B_{9r/10}}\int_{|x-y|>r}\frac{u^2(x)(\eta(x)-\eta(y))^2}{|x-y|^{n+2s}}&\leq C \sup_{B_r} u^2\int_{B_{9r/10}}\int_{|x-y|>r}\frac{1}{|x-y|^{n+2s}} \leq Cr^{n-2s}\sup_{B_r} u^2.\nonumber
\end{align}
For the cross term involving $ u(x)(\eta(x)-\eta(y))(u(x)-u(y))$, H\"older inequality implies
\begin{align}
\int_{\RR^n \setminus {B_{9r/10}}}\int_{B_{9r/10}} &\frac{u(x)(\eta(x)-\eta(y))(u(x)-u(y))}{|x-y|^{n+2s}} \\&\leq a \sup_{B_r} u \int_{\RR^n \setminus {B_{9r/10}}}\int_{B_{9r/10}}\frac{(\eta(x)-\eta(y))^2}{|x-y|^{n+2s}} + \frac{1}{a}\sup_{B_r} u \int_{\RR^n \setminus {B_{9r/10}}}\int_{B_{9r/10}}\frac{(u(x)-u(y))^2}{|x-y|^{n+2s}}. \nonumber
\end{align}
Choose $a$ small, such that the second term is absorbed in the negative terms in \eqref{3.25}. Finally we estimate the $(\eta(x) - \eta(y))^2$ term as before by splitting the part $|x-y|<r$ and $|x-y|>r$,
\begin{align}
\int_{\RR^n \setminus {B_{9r/10}}}\int_{B_{9r/10}} \frac{(\eta(x)-\eta(y))^2}{|x-y|^{n+2s}} &\leq \int_{|x-y|<r}\int_{B_{9r/10}} \frac{(\eta(x)-\eta(y))^2}{|x-y|^{n+2s}}  + \int_{|x-y|>r}\int_{B_{9r/10}} \frac{(\eta(x)-\eta(y))^2}{|x-y|^{n+2s}} \\&\leq
Cr^{-2}\int_{|x-y|<r}\int_{B_{9r/10}} \frac{|x-y|^{2}}{|x-y|^{n+2s}} +\int_{|x-y|>r}\int_{B_{9r/10}} \frac{1}{|x-y|^{n+2s}}\\& \leq Cr^{n-2s}.\nonumber
\end{align}
Combining all the estimates above, we have 
\begin{equation}
\Lambda |\{u>0\}\cap B_{3r/4}| \leq Cr^{n-2s}\sup_{B_r} u\nonumber
\end{equation}

We prove the following equivalent statement of the Theorem: for any $r >0$ small, there exists $r_0 > r$, such that \eqref{nondetor} is true. Suppose for contradiction that there exists $r > 0$, there exists $\eps < \eps_0$, such that for all $R > r$,
\begin{equation}
\sup_{B_R} u \leq \eps R^s. \nonumber
\end{equation} 
We claim that
\begin{equation}
\sup_{B_{r/2}(0)}u < \frac{\eps}{2}\left(\frac{r}{2}\right)^s. \nonumber
\end{equation}
For $y \in B_{r/2}(0)$, $x_0 \in B_{r/4}(y) \subset B_r(0)$, by Theorem 1.2\cite{Kuusi2015}, we have:
\begin{align} \label{kms}
|u(x_0)| &\leq \int_0^{r/4}\left(\frac{|B_t(x_0)|}{t^{n-2s}}\frac{1}{t}\right)dt + C\left(\frac{1}{|B_{r/4}(x_0)|}\int_{B_{r/4}(x_0)}|u(x)|dx\right) 
\\&\quad + C\left(\left(\frac{r}{4}\right)^{2s}\int_{\mathcal{R}^n_k\setminus B_{r/4}(x_0)}\frac{|u(x)|}{|x-x_0|^{n+2s}}dx\right).
\end{align}
We estimate the three terms above respectively:
\begin{align}
1.)&\int_0^{r/4}\left(\frac{|B_t(x_0)|}{t^{n-2s}}\frac{1}{t}\right)dt = C\left(\frac{r}{4}\right)^{2s}\\
2.)&\left(\frac{1}{|B_{r/4}(x_0)}\int_{B_{r/4}(x_0)}|u(x)|dx\right) \leq C\frac{1}{r^n}|\{\{u>0\} \cap B_{\frac{3}{4}r}(x_0)\}|\sup_{B_r(x_0)}u \leq C\eps^2. \nonumber
\end{align}
For the third term, for any $R>r_0$,
\begin{align}
3.)&\left(\left(\frac{r}{4}\right)^{2s}\int_{\mathcal{R}^n_k\setminus B_{r/4}(x_0)}\frac{|u(x)|}{|x-x_0|^{n+2s}}dx\right) \\
&\quad\leq C\left(\frac{r}{4}\right)^{2s}\left(\int_{\mathcal{R}^n_k\setminus B_{R}(x_0)}\frac{|u(x)|}{|x-x_0|^{n+2s}}dx + \int_{B_R(x_0)\setminus B_{r/4}(x_0)}\frac{|u(x)|}{|x-x_0|^{n+2s}}dx\right).\nonumber
\end{align}
The first term in the sum above is dominated by 
\begin{equation}
C\left(\frac{r}{4}\right)^{2s}\|u\|_{L^\infty(\mathcal{R}_k^n)}R^{-2s} \nonumber
\end{equation}
and is small by taking $R$ big. For the second term, we have
\begin{align} \label{tail}
C\left(\frac{r}{4}\right)^{2s}&\int_{B_R(x_0)\setminus B_{r/4}(x_0)}\frac{1_{\{B_R(x_0)\cap \{u(x)>0\}\}}|u(x)|}{|x-x_0|^{n+2s}}dx \\&
\leq  C\left(\frac{r}{4}\right)^{2s}\left(\int_{r/4}^R \int_{\D B_t(x_0)}\frac{Mt^s1_{\{B_R(x_0)\cap \{u(x)>0\}\}}}{t^{n+2s}}dS_xdt\right)\\
&= C\left(\frac{r}{4}\right)^{2s}\left(\int_{r/4}^R \frac{Mt^{s}}{t^{n+2s}} \int_{\D B_t(x_0)}1_{\{B_R(x_0)\cap \{u(x)>0\}\}} dS_xdt\right).
\end{align}
We integrate by part in $t$. The boundary term is given by 
\begin{align}
\eps t^{-n-s}\int_{B_t}1_{\{u>0\}}dV_x|^R_{r/4} \leq C\eps t^{-n-s}Mt^{n+s}|^R_{r/4}
\leq C\eps^2. \nonumber
\end{align}
The other term is 
\begin{align}
-\int_{r/4}^R \eps(-n-s)t^{-n-s-1}\int_{B_t(x_0)}1_{\{B_R(x_0)\cap \{u(x)>0\}\}} dV_xdt &\leq C\int_{r/4}^R \eps(n+s)t^{-n-s-1}\eps t^{n-s}dt\\
&=C\eps^2 \left(\left(\frac{r}{4}\right)^{-2s}-R^{-2s}\right). \nonumber
\end{align}
Combining the above gives:
\begin{align}
\eqref{tail} \leq C\left(\frac{r}{4}\right)^{2s}\eps^2\left(\left(\frac{r}{4}\right)^{-2s}-R^{-2s}\right). \nonumber
\end{align}
Therefore, taking $\eps_0$ and $r_0$ small establishes the claim.
Now we iterate to get:
\begin{equation}
\sup_{B_{\frac{r}{2^k}}(x_0)} u \leq C2^{1-k-ks}r^{s}. \nonumber
\end{equation}
This gives $u(x_0) = 0$, which concludes the proof.
\end{proof}
\begin{thm} \label{bounded}
Assume $A$ is $\lambda_k$-minimizer of \eqref{problem}, then $A$ is bounded in $\mathcal{R}_k^n$.
\end{thm}
\begin{proof}
By Theorem \ref{minissub}, $A$ is a local shape subsolution for the torsion energy problem. Let $u$ minimizes the corresponding torsion energy $E(A)$. Suppose $A_i = A|_{\RR^n}$ is not bounded, then there exists a sequence $x_l \in A_i$ and $|x_l| \rightarrow \infty$, and the distance between any two of them is $\geq 2R$ for $R$ to be chosen. We can also assume that $u(x_l) \neq 0$ in the measure theoretic sense. By previous lemma, we know that the exists $r_0$ such that
\begin{equation}
\sup_{B_{r_0}(x_l)}u \geq C_0r_0^s. \nonumber
\end{equation}
So there exists $y_l \in B_{r_0}(x_l)$ such that $u(y_l) \geq C_0r_0^s$. For any $r > 0$, let $\{\psi_l\}$ be a partition of unity in $\RR^n$ with supports in the balls $B_{r}(y_l)$. 
Using 
\eqref{kms}, we have
\begin{align}
C_0r_0^s \leq u(y_l) \leq C\int_0^r \left(\frac{|B_t(y_l)|}{t^{n-2s+1}}\right)dt &+ \frac{1}{|B_r(y_l)|}\int_{B_r(y_l)}|u(x)|dx \\&+ Cr^{2s}\int_{\mathcal{R}^n_k \setminus B_r{(y_l)}}\frac{|u(x)|}{|x-y_l|^{n+2s}}dx.\nonumber
\end{align}
The first two terms on the right are estimated as before, and we have:
\begin{equation}
C\int_0^r \left(\frac{|B_t(y_l)|}{t^{n-2s+1}}\right)dt + \frac{1}{|B_r(y_l)|}\int_{B_r(y_l)}|u(x)|dx \leq C(r^{2s} + C_0^2).\nonumber
\end{equation}
This term is absorbed by the left hand side for a sufficiently small $C_0$ by taking $r$ small. For the third term, we have for any $r >0$,
\begin{align}
Cr^{2s}\int_{\mathcal{R}^n_k \setminus B_r(y_l)}\frac{|u(x)|}{|x-y_l|^{n+2s}}dx &\leq Cr^{2s}\int_r^R \int_{\D B_t}\frac{|u(y_l+tw)|}{|x-y_l|^{n+2s}}dS_wdt+ Cr^{2s}\int_R^\infty \int_{\D B_t}\frac{|u(y_l+tw)|}{|x-y_l|^{n+2s}}dS_wdt
\\& \leq Cr^{2s}|\{u>0\}\cap B_R(y_l)|\int_r^R \frac{t^{n-1}}{t^{n+2s}}dt+ Cr^{2s}R^{-2s}\|u\|_{L^\infty}\\&\leq C|\{u>0\}\cap B_R(y_l)|+ Cr^{2s}R^{-2s}\|u\|_{L^\infty}\nonumber
\end{align}
where $\|u\|_{L^\infty} < C$ and the second term above can be absorbed by $C_0r_0$ by taking $R$ large. Therefore, we have:
\begin{equation}
Cr_0^s \leq |\{u>0\}\cap B_R(y_l)|.\nonumber
\end{equation}
Summing over $l$ shows that 
\begin{equation}
|A| \geq |A^i| \geq \sum_{l=1}^\infty |\{u>0\}\cap B_R(y_l)| = \infty, \nonumber
\end{equation}
which is a contradiction to $A$ being a set of finite measure.
\end{proof}

\section{Existence of Shape Minimizers}\label{Section:Exist}

Let $\Omega = B_R$ be a ball with radius $R$ in $\RR^n$. We first show that $\lambda_k$-minimizers over sets in $B_R$ exists, and then show that we can take $R \rightarrow \infty$. The strategy is to first construct $A$ as the positivity set of the functions $u_i$ such that the minimum below is obtained, and then show that this $A$ is the $\lambda_k$-minimizer of problem \eqref{problem}. Let
\begin{align}\label{2.1}
\tilde{\lambda}_k = \inf_{u_i\in \tilde{W}^{s,2}(\Omega)} \max_{\alpha\in \RR^k\setminus \{0\}} \left(R[\sum_1^k \alpha_i u _i]\ s.t.\  \sum_1^k |\{u_i\neq 0\}|\leq C, (u_i,u_j)_{W^{2,s}}= \delta_{ij}\right).
\end{align}
\begin{align}\label{2.3}
\lambda_k(A) = \inf_{u_i\in \tilde{W}^{s,2}(\Omega)} \max_{\alpha\in \RR^k\setminus \{0\}} \left(R[\sum_1^k \alpha_i u_i]\ s.t.\  u_i=0\ a.e. \text{ in }\RR^n\setminus A, (u_i,u_j)_{W^{2,s}}= \delta_{ij}\right).
\end{align}
Here $R[\cdot]$ is the Rayleigh quotient \eqref{rayleighrn}.
\begin{lem}
The minimum in \eqref{2.1} is attained for some $u \in \tilde{W}^{s,2}(\Omega)$ and $|\{u\neq 0\}|\leq C$.
\end{lem}
\begin{proof}
Take a minimizing sequence $R(u^l) \rightarrow \tilde{\lambda}_k$,
\begin{equation}
u^l = \frac{\sum_1^k  \alpha^l_iu_i^l}{\|\sum_1^k  \alpha^l_iu_i^l\|_{L^2}}.\nonumber
\end{equation} 
Since $R(u^l)$ is bounded, and that $\|u^l\|_{L^2} = 1$,  we know ${u^l}$ is bounded in $\tilde{W}^{s,2}(\Omega)$. Thus $u_l \rightarrow u$ for some $u$ weakly in $\tilde{W}^{s,2}(\Omega)$. By compact embedding, $u^l \rightarrow u$ strongly in $L^2(\Omega)$. Thus, there exists a subsequence $u^{l} \rightarrow u$ pointwise a.e. Since the volume of level sets are lower-semicontinuous with respect to pointwise convergence, we have
\begin{equation}
|\{u>0\}|+|\{u<0\}| \leq \liminf(|\{u^{l}>0\}|+|\{u^{l}<0\}|) \leq C.\nonumber
\end{equation}
\end{proof}

\begin{thm}
The $\lambda_k$-minimizer of problem \eqref{problem} exists in the sense of Definition \ref{defmini}.
\end{thm}
\begin{proof}
Take $\Omega = B_R$, we first show that the minimizer exists over sets contained in $B_R \subset \RR^n$. Let $u_R = \sum_1^k  \alpha_iu_i$ where the $\{u_i\}$'s are the minimizers in \eqref{2.1}, and let $A_R = \{u_R\neq 0\}$. We first observe that $\tilde{\lambda}_k \leq  \lambda_k(A_R)$ since $\{u:u=0$ a.e. in $\RR^n \setminus A_R\}\subset \{u:|\{u\neq 0\}| \leq C\}$. Moreover, since the the infimum in \eqref{2.1} is achieved by construction in \eqref{2.3}, we have
\begin{equation}
\tilde{\lambda}_k = \lambda_k(A_R), \nonumber
\end{equation}
that is, $u_R$ is the $k$-th eigenvalue of $A_R$ and $A_R$ is the minimizing set. 

We'd like to show that as $R \rightarrow \infty$, the shape minimizer $A_R$ converges in some sense. Indeed, by Theorem \ref{bounded} and the remark that follows, we can find $R_0 > 0$ such that for all $R > R_0$, we have $u_R^i \in \tilde{W}^{s,2}(\RR^n)$, $1\leq i \leq k$ such that 
\begin{enumerate}
\item $u_R = \sum_{i=1}^k u_R^i$.
\item $\{u_{R_0}^i \neq 0\}$ is bounded.
\item $u_{R}^i$ is a translation of $u_{R_0}^i$.
\item $\text{dist}(\{u_R^i \neq 0\},\{u_R^j\neq 0\}) \rightarrow \infty$ as $R\rightarrow \infty$ for all $i\neq j$.
\end{enumerate}

Now we consider $u_{R_0}^i \subset \RR^n_i$, and define
\begin{align}
u &= \sum_{i=1}^k u_{R_0}^i \subset \mathfrak{S}_k,\\
A &= \{u \neq 0\}\subset \mathcal{R}_k^n. \nonumber
\end{align}
We observe that $\lambda_k(A_R)\rightarrow \lambda_k(A)$, and $|A| = |A_R|$. 

We claim that $A$ is the $\lambda_k$-minimizer of \eqref{Problem}. If not, there exists $A' \subset \mathcal{R}_k^n$ such that $|A'| = |A|$ and $\lambda_k(A') < \lambda_k(A)$. By Theorem \ref{bounded} and the remark that follows, we can find a minimizing sequence $u'_R = \sum_{i=1}^k {u'_R}^{i}$  that satisfies (1)-(4) above. Let $A'_R = \{u'_R  \neq 0\}$, then $\lambda_k(A'_R)\rightarrow \lambda_k(A') < \lambda_k(A)$ and $|A'_R|=|A'|= |A|$, but this contradicts the minimality of $\lambda_k(A_R)$, which finished the proof of claim and the theorem.
\end{proof}
\begin{rem}
The argument above doesn't prove that $A$ has finitely many connected pieces, as in each copy or $\RR^n$, it might happen that $A_i$ consists of infinitely many connected pieces.
\end{rem}

\section{Optimal Regularity} \label{optimalregularity}

Let $A$ be a $\lambda_k$-minimizer of problem \eqref{problem}, $\lambda_k$ be the minimizing eigenvalue with multiplicity $l$, that is:
\begin{equation}
\lambda_k = \lambda_{k-1} = \cdots = \lambda_{k-l+1}.\nonumber
\end{equation}
Let $u_k,\cdots, u_{1}$ be the corresponding normalized eigenfunctions to $\lambda_k(A),\cdots, \lambda_1(A)$. Assume $0$ is a free boundary point, that is $0 \in \D A$. Let $h$ be the $s$-harmonic replacement of $u_k$ on the ball $B_r(0)$ and that $|B_r(0)\cap A^C| \neq 0$, that is, 
\begin{align} \label{hsharmonic}
\begin{cases}
(-\Delta)^s h = 0\ & \hbox{in}~B_r= B_r(0),\\
h = u & \hbox{in}~\mathcal{R}^n_k \setminus B_r(0),
\end{cases}
\end{align}

\begin{rem}\label{rembdryomega}
We note that including the set of interior nodal points 
\begin{equation}\label{intnodalpt}
\{x:\exists r>0, |B_r(x) \cap A^c| = 0, u_k(x) = 0\} \subset A
\end{equation}
doesn't change the $k$-th eigenvalue and corresponding eigenfunctions. Therefore, we take $A = A \cup \{B_r(x): |B_r(x) \cap A^c| = 0\}$. The points in \eqref{intnodalpt} are in the interior of $A$ and it is known that the eigenfunctions are smooth in the interior, which is shown, for example, in \cite{nonlocalbook},.
\end{rem}

\begin{prop} \label{4.10}
Let $A$ be a $\lambda_k$-minimizer of problem \eqref{problem}, and $\{u_{k-i}\}_{i=0}^{l-1}$ be the corresponding eigenfunctions of the minimizing $\lambda_k$, then there exist $\alpha_{k-l+1},\cdots, \alpha_{k}$ such that for 
\begin{equation}
u_\alpha = \alpha_{k-l+1}u_{k-l+1}+\cdots+ \alpha_{k}u_k,\ \|u_\alpha\|_{L^2} = 1, \nonumber
\end{equation}
we have
\begin{equation}
B[u_\alpha-h,u_\alpha-h] \leq Cr^n. \nonumber
\end{equation}
\end{prop}
\begin{proof}
By the minimality of $A$, we have
\begin{equation} \label{4.11}
\lambda_k(A) + |A| \leq \lambda_k(A\cup B_r) + |A \cup B_r| \Rightarrow \lambda_k(A)\leq \lambda_k(A\cup B_r) + Cr^n
\end{equation}
Let $w = h - u_k$, $\tilde{w} = w  - \sum_{i=1}^{k-1}\proj[u_i]{w}$, $\tilde{w}_j = \proj[u_j]{\tilde{w}}$ for $j = u-l+1,\cdots, k$. Here
\begin{equation}
\proj[u_i]{w} = B[w,u_i]u_i. \nonumber
\end{equation}
is the projection of $w$ in the direction $u_j$. Let 
\begin{equation}
w_\alpha = \sum_{j=k-l+1}^k\alpha_j \tilde{w}_j \nonumber
\end{equation}
for $\alpha$'s to be chosen.
Then by the min-max property of $\lambda_k(A\cup B_r)$, for any $t\in \RR$, we have
\begin{align}
\lambda_k(A\cup B_r) &\leq \max_{\alpha_1, ...\alpha_k}R[\sum_{i=1}^{k-l} \alpha_i u_{i} + \sum_{i=k-l+1}^k\alpha_j (u_j + t\tilde{w}_j)] \\&= R[\sum_{j=k-l+1}^k (\alpha_ju_j + t\alpha_j \tilde{w}_j)] \\&= R[u_\alpha + tw_\alpha] \nonumber
\end{align}
where $R$ is the Rayleigh quotient \eqref{rayleighrn}, and the max is achieved at $u_\alpha+tw_\alpha$ for some $\alpha_j, j = k-l+1,\cdots, k$, since it is orthogonal to $u_1,...,u_{k-1}$. We first claim that for any $t$ small,
\begin{equation}
R[u_\alpha + tw_\alpha] \leq B[u_\alpha + tw_\alpha,u_\alpha + tw_\alpha] + Cr^{n+1}\nonumber
\end{equation}
Indeed, since the denominator in the Rayleigh quotient 
\begin{equation}
\|u_\alpha + tw_\alpha\|^2_{L^2} = \|u_\alpha\|^2_{L^2} + t^2\|w_\alpha\|^2_{L^2} + 2t(u_\alpha,w_\alpha)_{L^2} = 1 + t^2\|w_\alpha\|^2_{L^2} + 2t(u_\alpha,w_\alpha)_{L^2}\nonumber
\end{equation}
and that 
\begin{equation}
t^2\|w_\alpha\|^2_{L^2} + 2t(u_\alpha,w_\alpha)_{L^2} = t\int_{B_r}w_\alpha(tw_\alpha+2u) \leq t\int_{B_r} \|w_\alpha (tw_\alpha+2u_\alpha) \|_{L^\infty}\leq Cr^{n+1}.\nonumber
\end{equation}
Therefore, we have 
\begin{align}
B[u_\alpha,u_\alpha] = R[u_\alpha] &\leq R[u_\alpha + tw_\alpha] + Cr^{n+\eps}
\\& \leq B[u_\alpha + tw_\alpha,u_\alpha + tw_\alpha] + Cr^{n+\eps}
\\& = B[u_\alpha,u_\alpha] + 2tB[u_\alpha,w_\alpha] + t^2 B[w_\alpha,w_\alpha] + Cr^{n+\eps}.\nonumber
\end{align}
Here we move the term linear in $t$ to the left, and that the $r^{n+1}$ term is dominated by the $r^{n+\eps}$ term. So, we have from \eqref{4.11}
\begin{align}
B[u_\alpha,u_\alpha] = R[u_\alpha] &\leq R[u_\alpha + tw_\alpha] + Cr^n
\\& \leq B[u_\alpha + tw_\alpha,u_\alpha + tw_\alpha] + Cr^n
\\& = B[u_\alpha,u_\alpha] + 2tB[u_\alpha,w_\alpha] + t^2 B[w_\alpha,w_\alpha] + Cr^n. \nonumber
\end{align}
Here we move the term linear in $t$ to the left. Since $t<0$, divide by $-t$ gives
\begin{equation} \label{4.14}
B[u_\alpha,w_\alpha]\leq |t|B[w_\alpha,w_\alpha] + C\frac{r^n}{|t|} \leq C_1(\alpha)|t|B[w,w] + C\frac{r^n}{|t|}
\end{equation}
where $C_1(\alpha) = \max \alpha_j^2$.
We observe that 
\begin{equation}
B[u_\alpha,w_\alpha] = B[u_\alpha,\sum_{j=k-l+1}^k \alpha_k \tilde{w}_j + \sum_{i = 1}^{k-1}\proj[u_i]{w}] \geq B[u_\alpha,C_2(\alpha)w] \nonumber
\end{equation} 
where $C_2(\alpha) = \{\min \alpha_j^2| \alpha_j \neq 0 \}$, since $u_\alpha$ is orthogonal to $u_1,...,u_{k-1}$.
Moreover, since $w = h-u_k$, integration by part gives $B[h,w] = 0$ since $h$ is $s$-harmonic. Therefore, $B[u_\alpha,w] = B[w,w]$ and \eqref{4.14} becomes
\begin{equation}
C_2B[w,w] \leq C_1|t|B[w,w] + C\frac{r^n}{|t|}. \nonumber
\end{equation}
Now choosing $|t|$ small yields
\begin{equation}
B[w,w]\leq Cr^n. \nonumber
\end{equation}
\end{proof}

\begin{rem}
Proposition \ref{4.10} shows one main difficulty in studying spectral minimizers as the estimates depend on a linear combination of $u_k, \cdots, u_{k-l}$ instead on a function $u_k$. How to recover some information on $u_k$ from estimates on the linear combination remains unknown.
\end{rem}

A immediate corollary is that 
\begin{cor} If $\lambda_k$ is simple with corresponding eigenfunction $u_k$, we have
\begin{equation}
B[u_k-h,u_k-h] \leq Cr^n. \nonumber
\end{equation}
\end{cor}

To prove the optimal regularity, we do extensions as in \cite{CaffarelliSilvestre} and reformulate the fractional Laplacian in $n+1$ dimension.
For $x\in \RR^{n+1}$, we write $x = (x',y)$, $x' \in \RR^n$. For $f' \in \tilde{W}^{2,s}(\RR^n)$, let $f: \RR^{n+1}\rightarrow \RR$ be defined as the following. First, for $y \geq 0$, we define $f$ to be the solution of 
\begin{equation}
\begin{cases}
\div(y^a \nabla f) = 0 & \hbox{in}~y>0,\\
f(x,0)= f'(x) & \hbox{on}~y=0. \nonumber
\end{cases}
\end{equation}
where $a = 1-2s$. We then extent $f(x,y)$ to the whole $\RR^{n+1}$ by setting $f(x,-y) = f(x,y)$.

By this construction, we can define extension of $\tilde{W}^{s,2}(\mathcal{R}^{n}_k)$ to $\tilde{W}^{s,2}(\mathcal{R}^{n+1}_k)$ by doing the extension in every copy of $\RR^n$ in $\mathcal{R}^n_k$. That is, for any $f' \in \tilde{W}^{s,2}(\mathcal{R}^n_k)$, define
\begin{equation}
f = \sum_{i=1}^k g,\nonumber
\end{equation}
where $g$ is the extension of ${f'|_{\RR^n_i}}$ from $\RR^n$ to $\RR^{n+1}$ defined above.
 
By \cite{CaffarelliSilvestre}, it is known that for $u' \in \tilde{W}^{2,s}({\mathcal{R}^n_k})$, let $u$ the its extension, then
\begin{equation}
\int_{\RR^n}\int_{\RR^n}\frac{(u'(x)-u'(y))^2}{|x-y|^{n+2s}}=\int_{\mathcal{R}^{n+1}_k} |x_{n+1}|^a|\nabla u|^2dx, \nonumber
\end{equation}

and thus problem \eqref{problem} is equivalent to minimizing the following functional in the extended domain

\begin{equation}\label{problemextend}
\frac{\int_{\mathcal{R}^{n+1}_k} |x_{n+1}|^a|\nabla u|^2}{\int_{\mathcal{R}^n_k \times \{0\}} |u|^2} + \int_{\mathcal{R}^n_k \times \{0\}} \chi_{\{u>0\}} +\chi_{\{u<0\}}\ d\mathcal{H}^{n-1}.
\end{equation}

Let $u$ to be the even extension of $u_\alpha$ in Proposition \ref{4.10}. 

By extension, we automatically have $\div (y^a \nabla u) = 0$ on $\mathcal{R}_k^{n+1} \setminus (\mathcal{R}_k^{n}\times \{0\})$.

We define the following space for $D\subset \mathcal{R}_k^{n+1}$,
\begin{align}
&L^2(a,D):= \{u \in L^2(D): |y|^{a/2}u \in L^2(D)\},\\
&H^1(a,D):= \{u \in L^2(D): |y|^{a/2} \nabla u \in L^2(D)\}. \nonumber
\end{align}

\begin{lem}\label{monotone} \cite{MarkAllen2012}
If $\div(|y|^a \nabla {h}) = 0$ in $B_R(x',0)$, then 
\begin{equation}
\frac{1}{r^{n+1+a}}\int_{B_r(x',0)}|y|^a|\nabla h|^2
\end{equation}
is monotone increasing in $r$ for $r< R$.
\end{lem}

The following optimal regularity uses technique similar to \cite{MarkAllen2012}. The proof is provided in the appendix for completeness. 

\begin{thm} \label{holders}
Let $u_\alpha$ be as in Proposition \ref{4.10} and $u$ its extension in $\mathcal{R}_k^{n+1}$, then $u \in C^{0,s}(\mathcal{R}_k^{n+1})$.
\end{thm}

The following results are immediate consequences of the above H\"older regularity. First we note that the $\lambda_k$-minimizers are open.

\begin{cor}
Let $A$ be a $\lambda_k$-minimizer of \eqref{problem}, then $A$ has an open representative.
\end{cor}
\begin{proof}
By definition, $A = \{u_\alpha \neq 0\}$. Since $u_\alpha$ is continuous, $A$ is open.
\end{proof}

We assume henceforth that the minimizing eigenvalue $\lambda_k$ in problem \eqref{problem} is simple. The next result is used to study the blow-up behavior in section \ref{sectionblowup}.

\begin{cor} \cite{MarkAllen2012}\label{urlim}
Let $\{u_p\}$ be a sequence of minimizers of  \eqref{problemextend} with $\|u_p\|_{L^\infty(\mathcal{R}_k^{n} \times \{0\})} \leq M$. Then There exists a subsequence  and a function $u_0$, such that for any open $U \subset\subset \mathcal{R}_k^{n+1}$, we have 
\begin{enumerate}
\item $u_0 \in H^1(a,U)\cap C^s(\overline{U})$.
\item $u_p \rightarrow u_0$ in $C^{\beta}$ for $\beta <s$.
\item $u_p \rightarrow u_0$ in $H^1(a,U)$ weakly.
\end{enumerate}
\end{cor}
\begin{proof}
Since $\|u_p\|_{L^\infty}$ is uniformly bounded, (1),(2) follows from Holder regularity. (3) follows from \eqref{4.35} in the proof of above Theorem in Appendix and Sobolev inequality.
\end{proof}

\begin{prop} \label{4.25} \cite{MarkAllen2012}
Let $u$ be a minimizer of \eqref{problemextend}, for any ball $B \subset \subset \RR^{n+1}$, then 
\begin{equation}
\int_B|x|^a|\nabla u|^2 = \int_{\D B}|x_n|^auu_{\nu}. 
\end{equation}
\end{prop}

\begin{notn}
For any set $D$ in $\RR^{n+1}$, we denote $D' = D|_{x_{n+1}=0}$.
For a function $f:\RR^{n+1}\rightarrow \RR$, denote $f' = f|_{x_{n+1}=0}$.
\end{notn}

\section{Non-degeneracy}\label{Section:nondegenerate}
In order to do blow up analysis, it is not a priori known that the blow up limit is non zero, so we need a lower bound for eigenfunctions. We note that in Theorem \ref{holders}, we only proved that some of the eigenfunctions are H\" older continuous. It is not known that the eigenfunctions that are H\"older continuous being non-degenerate. To overcome this issue, we henceforth assume that the $k$-th eigenvalue is simple. In this case, $u_k$ is H\"older continuous, and we proceed to show a lower bound for $u_k$. 

\begin{thm}\label{nondegeneracy}
For all $r_0>0$, there exists $r<r_0$, $c_0$, such that for $|B_{r/8}(x_0)\cap A| >0$, we have
\begin{equation}
\sup_{B_{r}(x_0)} |u_k| \geq c_0r^s. \nonumber
\end{equation}
\end{thm}
\begin{proof}

By translating, assume $x_0 = 0$. We proceed to show the following statement, which is equivalent to the statement of the theorem by scaling. 

For a fixed $r_0$, there exists $c_0,c_1 = c_1(c_0)$ to be chosen, independent of $r_0$ such that there exists $r < c_1r_0$, $\sup_{B_{r}}u \geq c_0r^s$. 

Suppose for contradiction that for all $r_0$, $c_0, c_1$, and for all $r< c_1r_0$, we have $\sup_{B_{r}}u \leq c_0r^s$. We claim that $\forall z\in B_{1/4}(x_0)$, 
\begin{equation} \label{5.4}
\sup_{B_{r/4}(z)}|u_{k}| \leq \frac{c_0r^s}{4}.
\end{equation}

Suppose for contradiction that there exists $r_0$ such that for all $r < r_0$, $\sup_{B_{r}(0)} u \leq c_0r^s$,
Applying the claim inductively to balls of radius $4^{-i}r$ centered on each $z \in B_{r/4}(0)$, we have that $|u_{k}|$ vanishes on $B_{r/4}$, which is a contradiction. Now we proceed to the proof of claim. 

Cacciopoli's inequality \eqref{Cacciopolli} implies
\begin{equation}\label{cacciopoli}
\int_{B_{9/10r}}\int_{B_{9/10r}} \frac{(u_k(x)-u_k(y))^2}{|x-y|^{n+2s}} < Cc_0r^{n-s}.
\end{equation}

Now, take $A' = A\setminus B_{3r/4}$ as a competitor for  $A$, then $|A\cap B_{3r/4}| \leq \lambda_k(A') - \lambda_k(A)$. Taking $\phi$ to be $1$ outside $B_{9r/10}$, vanishes on $B_{3r/4}$, and has the derivative bounded by $Cr^{-1}$. Let $v_j = \phi u_j$ for $j = 1,2,...,k$. We wish to control $\lambda_k(A')$ with $\lambda_k(A)$. We have
\begin{align}
&\int_{\RR^n}|v_k - u_k|^2 = \int_{B_{9r/10}}  |v_k - u_k|^2 \leq C\int_{B_{9r/10}} |u_k|^2 \leq Cc_0^2r^{2s}r^n.\nonumber
\end{align}

We first show that
\begin{align}\label{5.13}
\int_{\RR^n}\int_{\RR^n}\frac{((u_k-v_k)(x)-(u_k-v_k)(y))^2}{|x-y|^{n+2s}} \leq Cc_0r^{n-s}
\end{align}
To prove it, we break the integral, similar to before, into the following:
\begin{equation}
\int_{\RR^n}\int_{\RR^n} = \int_{B_{9r/10}} \int_{B_{9r/10}} + 2\int_{\RR^n \setminus B_{9r/10}}\int_{B_{9r/10}}:= I + II.\nonumber
\end{equation}
and we estimate 
\begin{align}
I &\leq 2\int_{B_{9r/10}} \int_{B_{9r/10}}\frac{(u_k(x)-u_k(y))^2 + (v_k(x)-v_k(y))^2)}{|x-y|^{n+2s}} 
\\& \leq Cc_0r^{n-s} + C\int_{B_{9r/10}} \int_{B_{9r/10}}\frac{(v_k(x)-v_k(y))^2}{|x-y|^{n+2s}}\\
&\leq  Cc_0r^{n-s} + \int_{B_{9r/10}} \int_{B_{9r/10}}\frac{((\phi(x)-\phi(y))u_k(x)-\phi(y)(u_k(x)-u_k(y)))^2}{|x-y|^{n+2s}}\\
&\leq  Cc_0r^{n-s} +C\int_{B_{9r/10}} \int_{B_{9r/10}} \frac{(\phi(x)-\phi(y))^2 \sup_{B_r}|u_k|^2 + \sup_{B_r}|\phi|^2(u_k(x)-u_k(y))^2}{|x-y|^{n+2s}}\\
&\leq  Cc_0r^{n-s} + Cc^2_0r^{2s}\int_{B_{9r/10}} \int_{B_{9r/10}} \frac{(\phi(x)-\phi(y))^2 }{|x-y|^{n+2s}} 
\\&= Cc_0r^{n-s} +  Cc^2_0r^{2s}\int_{|x-y|<r} \int_{B_{9r/10}} \frac{(\phi(x)-\phi(y))^2 }{|x-y|^{n+2s}} +  Cc^2_0r^{2s}\int_{|x-y|>r} \int_{B_{9r/10}} \frac{(\phi(x)-\phi(y))^2 }{|x-y|^{n+2s}} 
\\& \leq Cc_0r^{n-s} +  Cc^2_0r^{2s}r^{-2}\int_{|x-y|<r} \int_{B_{9r/10}} \frac{|x-y|^2 }{|x-y|^{n+2s}}+  Cc^2_0r^{2s}\int_{|x-y|>r} \int_{B_{9r/10}} \frac{1 }{|x-y|^{n+2s}}  \\&\leq Cc_0r^{n-s}.\nonumber
\end{align}
where we used \eqref{cacciopoli} for the terms involving $(u(x)-u(y))^2$, and we used $|\phi(x)-\phi(y)| < Cr^{-1}|x-y|$ when $|x-y|<r$, and $\phi \leq 1$ when $|x-y|>r$. For the $II$ term, using $\phi = 1$ on $\RR^n \setminus B_{9r/10}$, we have for the numerator in the integrand equals to
\begin{align}
(\phi(x)u_k(x)-\phi(y)u_k(y))^2 &- (u_k(x)-u_k(y))^2 \\&= (u_k(x)-u_k(y) + u_k(x)(\phi(x)-\phi(y)))^2 - (u_k(x)-u_k(y))^2
\\&= u_k^2(x)(\phi(x)-\phi(y))^2 + 2 u_k(x)(\phi(x)-\phi(y))(u_k(x)-u_k(y)).\nonumber
\end{align}
For the first term above, we estimate as before
\begin{align}
\int_{B_{9r/10}}\int_{\RR^n \setminus {B_{9r/10}}}\frac{u_k^2(x)(\phi(x)-\phi(y))^2}{|x-y|^{n+2s}} = &\int_{B_{9r/10}}\int_{|x-y|<r}\frac{u_k^2(x)(\phi(x)-\phi(y))^2}{|x-y|^{n+2s}} \\&+\int_{B_{9r/10}}\int_{|x-y|>r}\frac{u_k^2(x)(\phi(x)-\phi(y))^2}{|x-y|^{n+2s}}.\nonumber
\end{align}
Using that $|\phi(x)-\phi(y)| < Cr^{-1}|x-y|$ when $|x-y|<r$, and that $\phi \leq 1$ when $|x-y|>r$, we have
\begin{align}
\int_{B_{9r/10}}\int_{|x-y|<r}\frac{u_k^2(x)(\phi(x)-\phi(y))^2}{|x-y|^{n+2s}} &\leq C r^{-2}\sup_{B_r} u_k^2 \int_{B_{9r/10}}\int_{|x-y|<r}\frac{|x-y|^2}{|x-y|^{n+2s}} \leq Cc_0^2r^{n},\\
\int_{B_{9r/10}}\int_{|x-y|>r}\frac{u_k^2(x)(\phi(x)-\phi(y))^2}{|x-y|^{n+2s}}&\leq C \sup_{B_r} u_k^2\int_{B_{9r/10}}\int_{|x-y|>r}\frac{1}{|x-y|^{n+2s}} \leq Cc_0^2r^{n}.\nonumber
\end{align}

For the cross term involving $ u_k(x)(\phi(x)-\phi(y))(u_k(x)-u_k(y))$, we split the part $|x-y|<r$ and $|x-y|>r$ as before,
\begin{align}
\int_{\RR^n \setminus {B_{9r/10}}}\int_{B_{9r/10}} &\frac{u_k(x)(\phi(x)-\phi(y))(u_k(x)-u_k(y))}{|x-y|^{n+2s}} \\&\leq C \sup_{B_r} u_k \int_{|x-y|<r}\int_{B_{9r/10}} \frac{u_k(x)(\phi(x)-\phi(y))(u_k(x)-u_k(y))}{|x-y|^{n+2s}}  \\&\quad + C \sup_{B_r} u_k\int_{|x-y|>r}\int_{B_{9r/10}} \frac{u_k(x)(\phi(x)-\phi(y))(u_k(x)-u_k(y))}{|x-y|^{n+2s}} \\&\leq
Cr^{-1}\sup_{B_r}u_k\int_{|x-y|<r}\int_{B_{9r/10}} \frac{|x-y|^{1+s}}{|x-y|^{n+2s}}
\\&\quad +C\sup_{B_r}u_k\int_{|x-y|>r}\int_{B_{9r/10}} \frac{1}{|x-y|^{n+2s}}\\& \leq Cc_0 r^{n-s}.\nonumber
\end{align}
where the second inequality uses that $u_k \in C^{0,s}$. Combining all the estimates above, we proved claim \eqref{5.13}.\\

Next, we claim that
\begin{equation} \label{5.8}
\int_{\RR^n}\int_{\RR^n} \frac{(v_j(x)-v_j(y))^2}{|x-y|^{n+2s}} \leq\int_{\RR^n}\int_{\RR^n} \frac{(u_j(x)-u_j(y))^2}{|x-y|^{n+2s}} + C\sup_{B_r}|u_j|r^{n-2s}.
\end{equation}
for all $j = 1,\cdots,k$.

To prove the claim, the left hand side satisfies
\begin{align}
\int_{\RR^n}\int_{\RR^n} \frac{(\phi u_j(x)-\phi u_j(y))^2}{|x-y|^{n+2s}} =  &\int_{\RR^n}\int_{\RR^n} \frac{(u_j(x)-u_j(y))^2}{|x-y|^{n+2s}} + \frac{((u_j-v_j)(x)+(u_j-v_j)(y))^2}{|x-y|^{n+2s}} 
\\&\quad + \frac{2(u_j(x)-u_j(y)((u_j-v_j)(x)-(u_j-v_j)(y))}{|x-y|^{n+2s}}. \nonumber
\end{align}
By \eqref{5.13}, the first two terms above satisfies the desired inequality, and it is left to estimate the third term. Let $\psi = 1 - \phi$, so $u - v = \psi u$.

\begin{equation}
\psi(x)u_j(x) - \psi(y)u_j(y) = u_j(x)\psi(x)-\psi(y)) + \psi(y)(u_j(x)-u_j(y)) \nonumber
\end{equation}
so,
\begin{equation}
(u_j(x)-u_j(y))(\psi(x) u_j(x) - \psi(y)u_j(y)) =\psi(y)(u_j(x)-u_j(y))^2 + (u_j(x)-u_j(y))u_j(x)(\psi(x)-\psi(y)). \nonumber
\end{equation}
Therefore, 
\begin{align}
\int_{B_{9r/10}}\int_{B_{9r/10}}&\frac{2(u_j(x)-u_j(y))((u_j-v_j)(x)-(u_j-v_j)(y))}{|x-y|^{n+2s}} \\&= \int_{B_{9r/10}}\int_{B_{9r/10}}\frac{\psi(y)(u_j(x)-u_j(y))^2}{|x-y|^{n+2s}} + \frac{(u_j(x)-u_j(y))u_j(x)(\psi(x)-\psi(y))}{|x-y|^{n+2s}}\\
&\leq C\sup_{B_r}|u_j|r^{n-2s}+ C\sup_{B_{9r/10}}|u_j| \int_{B_{9r/10}}\int_{B_{9r/10}}\frac{(u_j(x)-u_j(y))^2 + (\phi(x)-\phi(y))^2}{|x-y|^{n+2s}}\\
&\leq C\sup_{B_r}|u_j|r^{n-2s}, \nonumber
\end{align}
and the other piece:
\begin{align}
\int_{\RR^n \setminus B_{9r/10}}\int_{B_{9r/10}}&\frac{2(u_j(x)-u_j(y))((u_j-v_j)(x)-(u_j-v_j)(y))}{|x-y|^{n+2s}} \\&=\int_{\RR^n \setminus B_{9r/10}}\int_{B_{9r/10}}\frac{\psi(y)(u_j(x)-u_j(y))^2}{|x-y|^{n+2s}} + \frac{(u_j(x)-u_j(y))u_j(x)(\psi(x)-\psi(y))}{|x-y|^{n+2s}}\\
&= \int_{\RR^n \setminus B_{9r/10}}\int_{B_{9r/10}}\frac{(u_j(x)-u_j(y))u_j(x)(\psi(x)-\psi(y))}{|x-y|^{n+2s}}\\
&\leq \sup_{B_{9r/10}}|u_j|\int_{|x-y|<r}\int_{B_{9r/10}}\frac{(u_j(x)-u_j(y))(\psi(x)-\psi(y))}{|x-y|^{n+2s}} 
\\&\hspace{0.3 in}+\sup_{B_{9r/10}}|u_j|\int_{|x-y|>r}\int_{B_{9r/10}}\frac{(u_j(x)-u_j(y))(\psi(x)-\psi(y))}{|x-y|^{n+2s}}
\\&\leq C\sup_{B_r}|u_j|\int_{|x-y|<r}\int_{B_{9r/10}} \frac{(u_j(x)-u_j(y))^2 + Cr^{-2}|x-y|^2}{|x-y|^{n+2s}} \\&\hspace{0.3 in}+ C\sup_{\RR^n}|u_j|\int_{|x-y|>r}\int_{B_{9r/10}}\frac{1}{|x-y|^{n+2s}}\\
&\leq C\sup_{B_r}|u_j|r^{n-2s}. \nonumber
\end{align}
This finishes the proof of claim \eqref{5.8}. 

We notice that $v_i$'s are not necessarily orthogonal and thus we need to orthogonalize $v_i$'s and we claim that the resulting $\tilde{v}_i$ also satisfy the inequality \eqref{5.8}. Let $\tilde{v}_k = v_k$, then define 
\begin{equation}
\tilde{v}_{i} = v_i - \sum_{j= i+1}^k\frac{B[v_i,\tilde{v}_{j}]}{B[\tilde{v}_{j},\tilde{v}_{j}]}\tilde{v}_{j} = v_i - \sum_jc_{ij}\tilde{v}_{j},\ i= k-1,k-2,...,1. \nonumber
\end{equation}
We have,
\begin{align}\label{Btildes}
B[\tilde{v}_i,\tilde{v}_i] &= B[v_i - \sum_j c_{ij}\tilde{v}_{j},v_i - \sum_j c_{ij}\tilde{v}_{j}]\\
&= B[v_i,v_i] + B[\sum_j c_{ij}\tilde{v}_{j},\sum_j c_{ij}\tilde{v}_{j}] -2 B[v_i,\sum_j c_{ij}\tilde{v}_{j}]
\\&= B[v_i,v_i] + \sum_j B[c_{ij}\tilde{v}_{j},c_{ij}\tilde{v}_{j}] - 2 B[\tilde{v}_i + \sum_j c_{ij}\tilde{v}_{j},\sum_j c_{ij}\tilde{v}_{j}]\\
&=B[v_i,v_i] + \sum_j B[c_{ij}\tilde{v}_{j},c_{ij}\tilde{v}_{j}] - 2\sum_j B[c_{ij}\tilde{v}_{j},c_{ij}\tilde{v}_{j}]
\\&=B[v_i,v_i] - \sum_j B[c_{ij}\tilde{v}_{j},c_{ij}\tilde{v}_{j}]\\
&\leq B[v_i,v_i]. \nonumber
\end{align} 
Combining the above with \eqref{5.8}, we have
\begin{equation}\label{5.24}
\int_{\RR^n}\int_{\RR^n}\frac{(\tilde{v}_i(x) - \tilde{v}_i(y))^2}{|x-y|^{n+2s}} \leq \int_{\RR^n}\int_{\RR^n} \frac{(u_i(x)-u_i(y))^2}{|x-y|^{n+2s}} + C\sup_{B_r}|u_i|r^{n-2s}.
\end{equation}
We also note that
\begin{equation}
\tilde{v}_i = \sum_{j = i}^k a_{ij}v_j.\nonumber
\end{equation}
Since
\begin{align}
&\int|v_j|^2 = \int|(1-(1-\phi))u_j|^2 \geq \int |u_j|^2 - \int_{B_{9r/10}}(1-\phi)u_j^2 \geq 1 - Cr^n, \nonumber
\\&\int v_jv_l = \int_{B_{9r/10}} (1-\phi^2)u_ju_l  \geq -Cr^n, \text{ for } j\neq l,
\end{align}
we have
\begin{align}\nonumber
\int|\tilde{v}_i|^2 &= \sum_{j = 1}^k a^2_{ij} \int |v_j|^2 + \sum_{j \neq l} a_{ij}a_{il}\int v_j v_l \geq \sum_{j = 1}^k a^2_{ij}(1-Cr^n) - Cr^n \geq 1 - Cr^n
\end{align}
and thus
\begin{align} \nonumber
\frac{1}{\int |\tilde{v}_i|^2} \leq 1 + Cr^{n}.
\end{align}
Therefore, for $r$ small,
\begin{equation}
\nonumber R[\tilde{v}_i] \leq (1+Cr^n)R[{v}_i] < R[{v}_k]
\end{equation}
Moreover, we have
\begin{equation}\label{5.27}
\int|v_k|^2 = \int|(1-(1-\phi))u_k|^2  \geq 1 - Cc_0^2r^{n+2s}.
\end{equation}

We note that the proof of \eqref{5.8} also show that 
\begin{equation}
\int_{\RR^n}\int_{\RR^n} \frac{(v_j(x)-v_j(y))^2}{|x-y|^{n+2s}} \geq \int_{\RR^n}\int_{\RR^n} \frac{(u_j(x)-u_j(y))^2}{|x-y|^{n+2s}} - Cc_0r^{n-s}. \nonumber
\end{equation}
Therefore, for $r$ small, we have
\begin{align}
\lambda_k(A') &\leq \max_{\alpha} R[\sum_{i=1}^k \alpha_i \tilde{v}_i]\\
&= \max_\alpha \frac{1}{|\sum_{i=1}^k \alpha_i \tilde{v}_i|_{L^2}}\int\int \frac{((\alpha_i \tilde{v}_i)(x) - (\alpha_i \tilde{v}_i)(y))^2}{|x-y|^{n+2s}}
\\&\leq \frac{1}{\int|v_k|^2}\int_{\RR^n}\int_{\RR^n} \frac{( {v}_k(x)- {v}_k(y))^2}{|x-y|^{n+2s}}\\
&\leq\int_{\RR^n}\int_{\RR^n} \frac{( u_k(x)- u_k(y))^2}{|x-y|^{n+2s}} + Cc_0r^{n-s}\\
&= \lambda_{k}(A) + Cc_0r^{n-s}. \nonumber
\end{align}
Therefore, from the minimizing property of $A$, we have
\begin{equation}
|A \cap B_{3r/4}| \leq Cc_0r^{n-s}. \nonumber
\end{equation}

We note that applying the above argument to $B_{c_1r}$, yield
\begin{equation}\nonumber
|A \cap B_{c_1r}| \leq Cc_0c_1^{n-s}r^{n-s}.
\end{equation}

Finally, take any $y\in B_{r/2}$, using \eqref{kms}, we have
\begin{align}
|u_k(y)|\leq CW^{\lambda_k u_k}_{s,2}(y,r/4)+ C\fint_{B_{r/4}(y)}|u|dx + C \text{Tail}(u_k;y,r/4). \nonumber
\end{align}
where
\begin{align}\nonumber
W^{\lambda_k u_k}_{s,2}(y,r/4) &= \int_0^\frac{r}{4}\left(\frac{|\lambda_k u_k|(B_\rho(y))}{\rho^{n-2s}}\right)\frac{1}{\rho}d\rho,\\
\text{Tail}(u_k;y,r/4) &= \left(\frac{r}{4}\right)^{2s}\int_{\RR^n\setminus B_{r/4(y)}}\frac{|u_k(x)|}{|x-y|^{n+2s}}dx.
\end{align}
We estimate the three terms individually,
\begin{align}
W^{\lambda_k u_k}_{s,2}(y,r/4) &= \int_0^\frac{r}{4}\left(\frac{|\lambda_k u_k|(B_\rho(y))}{\rho^{n-2s}}\right)\frac{1}{\rho}d\rho =\int_0^{r/4}\frac{1}{\rho \rho^{n-2s}}\int_{B_\rho(y)}|\lambda_k u_k|dxd\rho\\
&\leq C\int_0^{r/4}\frac{\rho^n}{\rho \rho^{n-2s}}\lambda_kc_0r^s \leq Cr^{2s}c_0r^s,\\
\fint_{B_{r/4}(y)}|u|dx &\leq \frac{C}{r^n}|B_{r/4}\cap A|c_0r^s \leq Cc_0^2,\\
\text{Tail}(u_k;y,r/4) &= \left(\frac{r}{4}\right)^{2s}\int_{\RR^n\setminus B_{r/4(y)}}\frac{|u_k(x)|}{|x-y|^{n+2s}}dx \\
&\leq C\left(\frac{r}{4}\right)^{2s}\left(\int_{\RR^n\setminus B_{c_1r}(y)}\frac{|u_k(x)|}{|x-y|^{n+2s}}dx + \int_{B_{c_1r}(y)\setminus B_{r/4}(y)}\frac{|u_k(x)|}{|x-y|^{n+2s}}dx\right). \nonumber
\end{align}
Here $c_1$ is some constant to be chosen large. The first term in the tail is dominated by 
\begin{equation}
Cc_1^{-2s}.\nonumber
\end{equation}
For the second term, using the assumption that $|u(x)| \leq c_0r^s$ on $B_r$ for $r < c_1r_0$, we have
\begin{align} 
C\left(\frac{r}{4}\right)^{2s}\int_{B_{c_1r}(y)\setminus B_{r/4}(y)}\frac{|u_k(x)|}{|x-y|^{n+2s}}dx &\leq Cc_0r^s\left(\frac{r}{4}\right)^{2s}\left(\frac{r}{4}\right)^{-n-2s}|A\cap B_{c_1r}| \leq Cc_0^2c_1^{n-s}. \nonumber
\end{align}

Combining all the above, we have
\begin{equation}
|u_k(y)| \leq C(c_0r^{3s}+c_0^2+c_1^{-2s} + c_0^2c_1^{n-s}). \nonumber
\end{equation}
Now we can choose $c_0, r$ small, $c_1$ large, to make the above $\leq 1/4c_0r^s$, which establishes the initial claim \eqref{5.4}.

\end{proof}

\section{Blow-up Analysis}\label{sectionblowup}
Let $u_k$ be the corresponding eigenfunction to the simple minimizing eigenvalue $\lambda_k$ and let $u$ be its extension. We have Weiss-type monotonicity formula.
\begin{thm}
Let $B_r = B_r(x_0,0)$. Then
\begin{align}\label{weissW}
W(r,u;x_0):= \frac{1}{r^n}\left(\int_{B_r}|y|^a|\nabla u|^2 + \int_{B_{r'}}\chi_{\{u\neq 0\}}\right) - \frac{s}{r^{n-1}}\int_{\D B_r}|y|^au^2,
\end{align}
is almost monotone, ie. $W'$ is a sum of positive terms and integrable terms.
\end{thm}
\begin{proof}
Without loss of generality, assume that $x_0 = 0$. Let $\tau_\eps = x + \eps \eta_p x$, where
\begin{equation}
\eta_p(x) = \max (0,\min(1,\frac{r-|x|}{k})). \nonumber
\end{equation}
Then $\eta_p(x) = 0$ outside $B_r(0)$, and
\begin{equation}
\eta_p	(x)\rightarrow \chi_{B_r(0)} \text{ as } p\rightarrow 0. \nonumber
\end{equation}
Let $\tau_\eps = x(1+\eps \eta_p(x))$, which leaves $\mathcal{R}^n_k \times\{0\}$ invariant. We can calculate
\begin{align}
&\nabla \eta_p(x) = \frac{-x}{|x|p}\chi_{B_r\setminus B_{r-p}},\\
&D\tau_\eps = I + \eps(\eta_p(x)I + x\nabla\eta_p(x)) + o(\eps),\\
&\det D\tau_\eps = 1 + \eps \text{trace}D(\eta_p(x)(x))+ o(\eps),\\
&\text{trace} D(\eta_p(x)x) = \div (\eta_p(x)(x)),\\
&D\tau^{-1}_\eps=I - \eps D(\eta_p(x)x)+o(\eps),\\
&\det D\tau^{-1}_\eps = 1 - \eps \text{trace} D(\eta_p(x)x) + o(\eps). \nonumber
\end{align}
We observe that since $u'$ is solution to our minimization problem, and thus the critical point of the functional
\begin{equation}
J[u]=R[u'] + |\{u'\neq 0\}| = \frac{\int_{\RR^n}\int_{\RR^n}\frac{(u'(x)-u'(y))^2}{|x-y|^{n+2s}}}{\|u'\|_{L^2(\mathcal{R}_k^n)}}+ |\{u'\neq 0\}| = \frac{\int_D |y|^a|\nabla u|^2}{\|u'\|_{L^2(\mathcal{R}_k^n)}}+ \int_{\mathcal{R}_k^n}\chi_{\{u'\neq 0\}}, \nonumber
\end{equation}
where $D \subset \mathcal{R}^{n+1}_k$ is an open set containing our domain. We note that by our normalization of the eigenfunction, $\|u'\|_{L^2(\mathcal{R}_k^n)} = 1$. Let $u_\eps(\tau_\eps(x)) = u(x)$ be our competitor, then we can calculate
\begin{align}
J[u_\eps] & = \frac{\int_D|y+\eps\eta_p(x)y|^a|\nabla u_\eps|^2d\tau_\eps(x)}{\int_{\mathcal{R}_k^n}|u'_\eps|^2d\tau'_\eps(x)} + \int_{\mathcal{R}_k^n}\chi_{\{u_\eps'\neq 0\}}d\tau'_\eps(x) 
\\&= \frac{\int_D|y+\eps\eta_p(x)y|^a|\nabla u(x)(D\tau_\eps)^{-1}|^2|D\tau_\eps|dx}{\int_{\mathcal{R}_k^n}|u'|^2|D\tau'_{\eps}|dx'}+\int_{\mathcal{R}_k^n}\chi_{\{u'\neq 0\}}|D\tau'_{\eps}|dx'\\
&=\frac{\int_D|y+\eps\eta_p(x)y|^a(|\nabla u|^2 - 2\eps\nabla u D(\eta_p(x)x)\nabla u)(1+\eps \div(\eta_p(x)x))}{\int_{\mathcal{R}_k^n}|u'|^2(1+\eps \div (\eta_p'(x',0)x') + o(\eps))dx'}
\\ &\quad \quad \quad +\int_{\mathcal{R}_k^n}\chi_{\{u'\neq 0\}}(1+\eps \div (\eta_p'(x',0)x') dx'+ o(\eps). \nonumber
\end{align}
Now we have
\begin{align} \label{4.43}
J[u_\eps]-J[u] &= \left(\frac{\int_D|y+\eps\eta_p(x)y|^a(|\nabla u|^2)}{\int_{\mathcal{R}_k^n}|u'|^2(1+\eps \div (\eta_p'(x',0)x') + o(\eps))dx'} - \int_D |y|^a|\nabla u|^2\right)
\\ &\quad + \eps\left( \frac{\int_D|y+\eps\eta_p(x)y|^a((|\nabla u|^2)(\div(\eta_p(x)x)) - 2\nabla u D(\eta_p(x)x)\nabla u)}{\int_{\mathcal{R}_k^n}|u'|^2(1+\eps \div (\eta_p'(x',0)x') + o(\eps))dx'}\right)\\
&\quad +\eps \int_{\mathcal{R}_k^n}\chi_{\{u'\neq 0\}}\div (\eta_p'(x',0)x')dx' + o(\eps).
\end{align}
Now letting $\eps$ be both positive and negative, divide by $\eps$, and the limit is $0$, 
\begin{equation}
\lim_{\eps\rightarrow 0}J[u_\eps]-J[u] = 0.
\end{equation}
The first line of \eqref{4.43} becomes the derivative
\begin{align}
\lim_{\eps\rightarrow 0}&\frac{1}{\eps}\left(\frac{\int_D|y+\eps\eta_p(x)y|^a(|\nabla u|^2)}{\int_{\mathcal{R}_k^n}|u'|^2(1+\eps \div (\eta_p'(x',0)x') + o(\eps))dx'} - \int_D |y|^a|\nabla u|^2\right) \\
&= \int_D a|y|^a\eta_p(x)|\nabla u|^2 + \int_D|y|^a|\nabla u|^2\left(\int_{\mathcal{R}_k^n}|u'|^2 \div(\eta_p'(x',0)x')\right). \nonumber
\end{align}
Thus, we obtain
\begin{align}
\int_D a|y|^a\eta_p(x)|\nabla u|^2 &+ \int_D|y|^a|\nabla u|^2\left(\int_{\mathcal{R}_k^n}|u'|^2 \div(\eta_p'(x',0)x')\right) \\&+ \int_D |y|^a|\nabla u|^2\div (\eta_p(x)x)-2 \nabla u D(\eta_p(x)x)\nabla u + \int_{\mathcal{R}_k^n}\chi_{\{u'\neq 0\}}\div(\eta_p'(x',0)x') = 0. \nonumber
\end{align}
Using
\begin{align}
\div (\eta_p(x)x) &= n\eta_p(x)- \frac{|x|}{p}\chi_{B_r \setminus B_{r-p}},\\
\div(\eta_p(x',0)x') &= (n-1)\eta_p' - \frac{|x'|}{p}\chi_{B_r' \setminus B_{r-p}'}, \nonumber
\end{align}
we have
\begin{align} \label{4.48}
(n-2+a)\int_{B_r}&|y|^a|\nabla u|^2\eta_p - \frac{1}{k}\int_{B_r\setminus B_{r-k}}|x||y|^a\left(|\nabla u|^2 - 2\left|(\nabla u,\frac{x}{|x|})\right|^2\right)\\
&+(n-1)\int_{B_{r}'}\chi_{\{u'\neq 0\}}\eta_p' - \frac{1}{k}\int_{B_r'\setminus B_{r-k}'}\chi_{\{u'\neq 0\}}|x'|\\
&+\int_D|y|^a|\nabla u|^2\int_{\RR^n}|u'|^2(n-1)\eta_p'  -   \int_D|y|^a|\nabla u|^2 \int_{B_r'\setminus B_{r-k}'}|u'|^2 \frac{|x'|}{k}=0.
\end{align}
For the latter two terms, since $u \in C^{0,s}(D)$ and $\int_D |y|^a|\nabla u|^2 = \lambda_k$, we have estimate as $p\rightarrow 0$:
\begin{align}
&\int_D|y|^a|\nabla u|^2\int_{\mathcal{R}_k^n}|u'|^2(n-1)\eta_p' \rightarrow \int_D|y|^a|\nabla u|^2\int_{B_r} |u'|^2(n-1) \\
&\quad \leq  \int_D|y|^a|\nabla u|^2\int_{B_r} r^{2s}(n-1) = Cr^{n+2s}
\\&\int_{B_r'\setminus B_{r-k}'}|u'|^2 \frac{|x'|}{k} \leq \frac{1}{k}\int_{B_r'\setminus B_{r-k}'}|x'|dx' = \frac{C}{k}(r^{n+2}-(r-k)^{n+2}) \rightarrow Cr^{n+1}. \nonumber
\end{align}
So, as $k\rightarrow 0$, \eqref{4.48} gives:
\begin{align}
(n-2+a)\int_{B_r}|y|^a|\nabla u|^2 &- r\int_{\D B_r}|y|^a(|\nabla u|^2 - 2u_\nu^2) + (n-1)\int_{B_r'}\chi_{\{u \neq 0\}} 
\\&- r\int_{\D B_r'}\chi_{\{u\neq 0\}} + Cr^{n+2s}+Cr^{n+1} = 0. \nonumber
\end{align}
From proposition \ref{4.25}, we have
\begin{equation}
\int_{B_r}|y|^a|\nabla u|^2 = \int_{\D B_r}|y|^auu_\nu. \nonumber
\end{equation}
Using the above, and divide both sides by $-r^{n+1}$ yields for almost every $r$,
\begin{align} \label{4.52}
\left[\frac{1}{r^{n}}\int_{B_r}|y|^a|\nabla u|^2\right]' &+ \left[\frac{1}{r^n}\int_{B_{r}'}\chi_{\{u \neq 0\}}\right]' 
\\&- \frac{1}{r^{n}}\int_{\D B_r}|y|^a\left(\frac{(1-a)uu_\nu}{r}-2u^2_\nu\right) + C(r^{2s-1}+1) = 0.
\end{align}
For $\eps < r$, we may integrate and use Fubini's theorem to get
\begin{align}
\int_\eps^r \frac{1}{\rho^{n-1+a}}\int_{\D B_\rho}|y|^a2uu_\nu dS d\rho &= \int_{\D B_1} \int_\eps^r |y|^a2u(\rho x)u_\nu(\rho x)d\rho dS 
\\&= \int_{\D B_1}|y|^n(u^2(rx)-u^2(\eps x))dS \\&= -c + \frac{1}{r^{n-1+a}}\int_{\D B_r}|y|^a u^2dS. \nonumber
\end{align}
So, for almost every $r$, we have
\begin{equation}
\frac{d}{dr}\left[\frac{1-a}{2r^{n-1}}\int_{\D B_r}|y|^a u^2\right] = \frac{1}{r^{n-2}}\int_{\D B_r}|y|^a\left(\frac{(1-a)uu_\nu}{r}- \frac{(1-a)^2u^2}{2r^2}\right). \nonumber
\end{equation}
We subtract and add the above term to \eqref{4.52} to get
\begin{align}
\left[\frac{1}{r^{n}}\int_{B_r}|y|^a|\nabla u|^2\right]' &+ \left[\frac{1}{r^n}\int_{B_{r}'}\chi_{\{u \neq 0\}}\right]' 
- \left[\frac{1-a}{2r^{n-1}}\int_{\D B_r}|y|^a u^2\right]' \\&- \frac{1}{r^{n-2}}\int_{\D B_r}|y|^a\left(\frac{(1-a)u}{\sqrt{2}r}- \sqrt{2}u_\nu\right)^2 + C(r^{2s-1}+1) = 0. \nonumber
\end{align}
We observe that $C(r^{2s-1}+1)$ is integrable in $r$. Thus, $W'$ is a sum of positive and integrable terms.
\end{proof}
\begin{rem}
Let $u_r = \frac{u(rx)}{r^s}$, then $W(1,u_r) = W(r,u)$. 
\end{rem}

One important consequence of the above almost monotonicity formula is that it shows that the blow up, defined as the limit of
\begin{equation}
u_r = \frac{u(rx)}{r^s} \rightarrow u_0 \nonumber 
\end{equation}
is homogeneous of degree $s$. We first show that $u_0$ satisfy a PDE on the compliment of $\Lambda(u)$.
\begin{lem} \label{u0pde} Let $u_0$ be the blow up, and after possibly passing to a subsequence, $u_r \rightarrow u_0$, we have
\begin{equation}
\div(y^a \nabla u_0) = 0 \text{ in } \mathcal{R}^{n+1}_k \setminus \{y = 0\}. \nonumber
\end{equation}
\end{lem}
\begin{proof}
Since $u_r \rightarrow u_0$ in $H^1(a,U)$, $u_0$ satisfy $\div(y^a \nabla u_0) = 0$ on $\mathcal{R}^{n+1}_k \setminus \{y = 0\}$. On the thin space $\mathcal{R}^{n}_k \times \{0\}$, we have 
\begin{equation}
(- \Delta)^s u_k = \lambda_k u_k. \nonumber
\end{equation}
Since
\begin{equation}
(- \Delta)^s u_k = -\lim_{y\rightarrow 0} y^a \D_y u(x,y). \nonumber
\end{equation}
So $u_r$ satisfies
\begin{equation}
(-\Delta)^s u_r' = r^{2s} \lambda_k u_r'.\nonumber
\end{equation}
Since $u_r' \rightarrow u_0'$ uniformly locally, we have $(-\Delta)^s u_0' = 0$ on $\{u_0' \neq 0\}$, and thus
\begin{equation}
-\lim_{y\rightarrow 0} y^a \D_y u_0(x,y) = 0. \nonumber
\end{equation}
So, the even reflection satisfies $\div(y^a \nabla u_0) = 0$ on $\mathcal{R}_k^n \times\{0\} \cap \{u_0 \neq 0\}$ (by \cite{CaffarelliSilvestre} Lemma 4.1).
\end{proof}

\begin{cor} \label{u0homs}
Let $u_0$ be the blow up at $(x_0,0)$, and after possibly passing to a subsequence, $u_r \rightarrow u_0$, then $u_0$ is homogeneous of degree $s$.
\end{cor}
\begin{proof}
By \eqref{4.35} and optimal regularity, $W(0+,u,x_0)$ is bounded from below. 
We claim that
\begin{equation}
W(r,u) = W(1,u_r)\rightarrow W(1,u_0) = W(0,u_r) = L.\nonumber
\end{equation}
and 
\begin{equation} \label{W(1,u_0)}
W(1,u_0) = \frac{1}{r^n}\left(\int_{B_1}|y|^a|\nabla u_0|^2 + \int_{B_{1}'}\chi_{\{u_0\neq 0\}}\right) - \frac{s}{r^{n-1}}\int_{\D B_1}|y|^au_0^2 = L.
\end{equation}
We first note that if the claim is true, then the above equation grantees that $u_0$ is homogeneous of degree $s$. We proceed to the proof of the claim. 

Taking $r\rightarrow 0$ in \eqref{weissW} we wish to show that the three terms converge to the corresponding limit respectively. By Corollary \ref{urlim}, the third term converges to the corresponding term in \eqref{W(1,u_0)}.
For the first term, let $\eta \in C_c^\infty$ with $\eta > 0$, and a sequence $\eps \rightarrow 0^+$ such that $\{u_0 \neq 0\}$ is a neighborhood of supp $\eta$. We have
\begin{equation}
\int |y|^a|\nabla u_0|^2 = \lim_{\eps \rightarrow 0}\int |y|^a\nabla u_0 \nabla (u_0-\eps)_+ \eta = -\lim_{\eps \rightarrow 0} \int |y|^a (u_0 -\eps)_+ \nabla \eta \nabla u_0 = -\int |y|^a u_0 \nabla \eta \nabla u_0. \nonumber
\end{equation}
where the first and last equality follows from dominated convergence theorem, and the second equality uses integration by part with respect to the measure $|y|^a dx$ and that $u_0$ is harmonic on $\{u_0 > \eps\}$. Note that above calculation also works if we replace $u_0$ by $u_r$, and thus 
\begin{equation}
\int |y|^a|\nabla u_0|^2 = -\int |y|^a u_0 \nabla \eta \nabla u_0 = -\lim_{r\rightarrow 0} |y|^a u_r \nabla \eta \nabla u_r  = \lim_{r\rightarrow 0} \int|\nabla u_r|^2 \eta. \nonumber
\end{equation}
This shows that $u_r \rightarrow u_0$ in $H^1_{\text{loc}}(a,U)$, and thus
\begin{equation}
\int_{B_1}|y|^a|\nabla u_r|^2 \rightarrow \int_{B_1}|y|^a|\nabla u_0|^2. \nonumber
\end{equation}

It is left to check that
\begin{equation}
\int_{B_{r}'}\chi_{\{u_r\neq 0\}} \rightarrow \int_{B_{1}'}\chi_{\{u_0\neq 0\}}. \nonumber
\end{equation}

First we note that $1_{\{u_r>0\}} \leq 1_{\{u_0>0\}}$ by lower semi-continuity and  $C^\beta$ convergence (Prop \ref{urlim}). We proceed to check the other inequality. 

Let $x$ be such that $u(x) = 0$, and there exists a sequence $x_r \rightarrow x$ such that $u_r(x_r) > 0$. Fix $\eps >0$, for all $r > 0$, there exists $y_r \in B_\eps(x)$ such that $u'_r(y_r)\geq c\eps^s$. Since $u \rightarrow u_0$ uniformly in compact sets, we have $y_r\rightarrow y_0 \in B_\eps(x)$ such that $u'_0(y_0) \geq c\eps^s$. Moreover, by H\"older continuity, there exists $\delta >0$, $B_{\delta\eps}(y_0) \subset B_\eps(x)$ such that  $u'_0(B_{\delta \eps}(y_0)) \geq c\eps^s$.
This implies that
\begin{equation}
\sup_{B_\eps(x)} u'_0 > C\eps^s, \text{ and } |\{u_0 > 0\} \cup B_r| > Cr^n, \nonumber
\end{equation}
that is, $x$ is in the measure theoretic closure of ${\{u'_0>0\}}$. By Lebesgue density theorem, the measure theoretic boundary satisfies $|\D_* \{u'_0 > 0\}| = 0$. The case where $x \in \{u_0<0\}$ follows from the same argument and the claim follows.
\end{proof}

\section{Separation of the Free Boundary and Density Estimates}\label{Section:separation}

We have established that if the minimizing eigenvalue $\lambda_k(A)$ of problem \eqref{problem} is simple, then the corresponding eigenfunction $u_k$ is $C^{0,s}$ and non-degenerate. In this section we follow the idea of \cite{MarkAllen2012} to show that the free boundary is separated.

Let $u_k$ be the corresponding eigenfunction to $\lambda_k$. Let $u(x,y) = u(x_1,\cdots,x_n,y)$ be the extension of $u_k$, $u_0$ be the blow up, and after possibly passing to a subsequence $u_r \rightarrow u_0$. Let 
\begin{equation}
\Lambda(u) = \mathcal{R}^{n}_k \times \{0\} \cap \{u = 0\}.
\end{equation}

\begin{lem}\cite{MarkAllen2012}
Let $u$ be homogeneous of degree $\alpha$ with nontrivial positive and negative parts, continuous, and satisfies $\div(|y|^a \nabla u) = 0$ in $\mathcal{R}^{n+1}_k \setminus \Lambda(u)$. Then $\alpha > \min\{1,1-a\}$.
\end{lem}
\begin{thm} \label{separation}
Let $A$ be a $\lambda_k$-minimizer, $u_k$ be the corresponding eigenfunction. Then 
\begin{equation}
\D\{u_k(x) > 0\} \cap \D\{u_k(x) < 0\} = \emptyset. \nonumber
\end{equation}
\end{thm}
\begin{proof}
If $u$ is not sign changing, the theorem is trivially true. Assume $u$ has both nontrial positive and negative part. Suppose by contradiction, there exists $x_0 \in \D\{u_k(x) > 0\} \cap \D\{u_k(x) < 0\}$. Let $u_0$ be the blow up limit, after possibly passing to a subsequence, $u_r = u(rx)/r^s \rightarrow u_0$. Lemma \ref{u0pde} shows that $\div(|y|^a\nabla u_0) = 0$. By the lower bound in Theorem \ref{bounded} and $C^\beta$ convergence, $u_0$ has non-trivial positive and negative parts. By Corollary \ref{u0homs}, $u_0$ is homogeneous of degree $s$, $s = (1-a)/2 < \min\{1,1-a\}$. This contradicts above lemma.
\end{proof}
It follows that we have the following density estimates. 
\begin{thm}
Let $A$ be a $\lambda_k$-minimizer, $u_k$ be the normalized corresponding eigenfunction. Suppose $x$ is a free boundary point, that is $x_0\in \D A$. Then, there exists $r_0$ such that for all $r<r_0$ small, we have
\begin{equation} \label{densityu>0}
|\{u_k > 0\} \cap B_r(x_0)| \geq Cr^n. 
\end{equation}
\begin{equation}\label{densityu=0}
|\{u_k = 0\} \cap B_r(x_0)| \geq Cr^n. 
\end{equation}
In particular, $A$ has no cusps. 
\end{thm}

\begin{proof}
First we note that \eqref{densityu>0} follows directly from optimal regularity and nondegeneracy. By Theorem \ref{nondegeneracy}, there exists $r < r_0$, $y \in B_{r/2}(x_0)$ such that $u(y) \geq c(r/2)^s$. Since $u \in C^{0,s}(B_r(x_0))$, there exists $c$ such that $u >0$ on $B_{cr}(y)$. Thus 
\begin{equation}
Cr^n = |B_{cr}(y)| \leq |B_{r}(x_0) \cap \{u>0\}|. \nonumber
\end{equation}

We proceed to prove \eqref{densityu=0}. Without loss on generality, suppose $x = 0$, $u_k \geq 0$ in $B_r(0) \cap A$. Suppose for contradiction, 
\begin{equation}
|B_r(0)\cap A_i^C| = o(r^n). \nonumber
\end{equation}
Let $h$ be the $s$-harmonic replacement of $u_k$ on $B_r(0)$, that is
\begin{align}
\begin{cases}
(-\Delta)^s h = 0\ & \hbox{in}~B_r= B_r(0),\\
h = u & \hbox{in}~\mathcal{R}^{n}_k \setminus B_r(0),
\end{cases}\nonumber
\end{align}
Using $u_k-h \in \tilde{W}^{s,2}_0(B_r)$ as a text function, we have
\begin{align}
B[h,u_k-h] &= \int_{\RR^n}\int_{{R}^{n}_k} \frac{(h(x)-h(y))((u_k-h)(x)-(u_k-h)(y))}{|x-y|^{n+2s}} = 0\\
B[u_k,u_k-h] &=\int_{{R}^{n}_k}\int_{{R}^{n}_k} \frac{(u_k(x)-u_k(y))((u-h)(x)-(u_k-h)(y))}{|x-y|^{n+2s}} \\&= \int_{B_r}\lambda_ku_k(u_k-h). \nonumber
\end{align} 
We first note that $h<u_k$ by strong maximum principle \cite{nonlocalbook} and thus $u_k-h > C$ in any compact set in $B_r$. Also by assumption, $|\{u_k>0\} \cap B_r(0)| = Cr^{n}$. Therefore,
\begin{equation}
\int_{B_r}\lambda_ku_k(u_k-h) \geq C\lambda_k\int_{\{u_k > 0\}\cap B_r}u_k > Cr^{n}, \nonumber
\end{equation}
and that 
\begin{equation}
B[u_k,u_k]-B[h,h] = B[u_k,u_k-h] + B[h,u_k-h] > Cr^{n}. \nonumber
\end{equation}
We also have that 
\begin{equation}
\|h\|_{L^2}^2 = \|u\|_{L^2}^2 + \int_{B_r}h^2 - \int_{B_r}u^2 \geq \|u\|_{L^2}^2 - Cr^{n+2s}. \nonumber
\end{equation}
where the last inequality uses $u \in C^{0,s}$. Therefore,
\begin{align}
R[h] - R[u_k] &= \frac{B[h,h]}{\|h\|_{L^2}} - {B[u,u]} < B[h,h]-B[u_k,u_k]+Cr^{n+2s}. \nonumber
\end{align}
Thus,
\begin{align}
\lambda_k(A\cup B_r) + |A\cup B_r| - (\lambda_k(A) + |A|) &\leq  R[h] - R[u_k] + Cr^{n+\eps} \\&< (B[h,h]-B[u_k,u_k]) + Cr^{n+2s} + o(r^n)\\
&\leq -Cr^n + Cr^{n+2s} + o(r^n) \\
&<0 \nonumber
\end{align}
for $r$ small. This contradicts the minimality of $A$.
\end{proof}

Above density estimate shows that the topological boundary and the measure theoretic boundary coincide and Lebesgue density theorem implies that the free boundary has measure $0$.
\begin{cor}
Let $A$ be a $\lambda_k$-minimizer, then $|\D A| = 0$.
\end{cor}

\section{Free Boundary Regularity}\label{freeboundary}
In this section we show that the minimizer $u_k$ of energy \eqref{problemextend} is an almost minimizer of the following Alt–Caffarelli–Friedman functional, then the free boundary regularity theorem from \cite{AllenGarcia} applies. Let
\begin{equation}\label{acfunc}
\tilde{J}(u,\Omega) = \int_{\Omega}|y|^a|\nabla u|^2 + \int_{\Omega \cap \{y=0\}}\chi_{\{u\neq 0\}}d\mathcal{H}^{n-1}.
\end{equation}
\begin{defn}
Let $\Omega \subset \mathcal{R}^{n+1}_k$ be an open set. Let $u$ be $a$-harmonic in $\Omega\setminus\{y = 0\}$ where $y = x_{n+1}$, with $u(x,y)=u(x,-y)$. We say $u$ is an almost minimizer of \eqref{acfunc} if there exists $\kappa$ and $\alpha$ such that 
\begin{equation}\label{almostmin}
\tilde{J}(u,B_r(x_0,0))\leq (1+\kappa r^\alpha)\tilde{J}(v,B_r(x_0,0))
\end{equation}
for any $v \in u+H_0^1(a,B_r(x_0,0))$.
\end{defn}

We recall the notation that for a function $f \in \mathcal{R}^{n+1}_k$, $f' = f|_{\{y = 0\}}$, and for a set $\Omega \in \mathcal{R}^{n+1}_k$, $\Omega ' = \Omega \cap \{y = 0\}$. 

\begin{lem}
Let $u_k$ be the $k$-th eigenfunction corresponding to the minimizing eigenvalue $\lambda_k(A)$, and $u$ be its extension, then $u$ is an almost minimizer of \eqref{acfunc}.
\end{lem}
\begin{proof}
Without loss of generality, let $\|u_k\|_{L^2(\mathcal{R}^{n}_k)} = 1$. We note that it is enough to check equation \eqref{almostmin} for $v$ where $v$ are local minimizers of $\tilde{J}(u,B_r(x_0,0))$, and $v = u$ in $\mathcal{R}_k^{n+1} \setminus B_r(x_0,0)$. It is known \cite{CaffarelliRoquejoffreSire} that $v\in C^{0,s}(K)$ for $K \subset\subset B_r$. Since $u$ minimizes \eqref{problemextend}, we have
\begin{align} \label{8.5}
\int_{\mathcal{R}^{n+1}_k}|y|^a|\nabla u'|^2 + \int_{\mathcal{R}^{n}_k \times \{0\}}\chi_{\{u\neq 0\}}d\mathcal{H}^{n-1} \leq \frac{\int_{\mathcal{R}^{n+1}_k}|y|^a|\nabla v|^2}{\int_{\mathcal{R}^{n}_k}|v'|^2} + \int_{\mathcal{R}^{n}_k \times \{0\}}\chi_{\{v\neq 0\}}d\mathcal{H}^{n-1}.
\end{align}
Noticing that 
\begin{align}
\int_{\{y = 0\}} |v'^2| &= \int_{\{y = 0\}} |v'^2 + u -u|\\
&= \int_{\{y = 0\}} |u'|^2 + |v'-u'|^2 + 2v(v'-u')\\
&= 1 + \int_{B_r\cap \{y=0\}} |v'-u'|^2 + 2v(v'-u')\\&
\leq Cr^{n+2s}. \nonumber
\end{align}
Here the last inequality follows from that $u,v\in C^{0,s}(K)$, for $K\subset\subset B_r(x_0,0)$.

The right hand side of \eqref{8.5} is thus bounded by
\begin{align}
(1+Cr^{n+2s}){\int_{\mathcal{R}^{n+1}_k}|y|^a|\nabla v|^2} &+ \int_{\mathcal{R}^{n}_k \times \{0\}}\chi_{\{v\neq 0\}}d\mathcal{H}^{n-1} \\&= (1+Cr^{n+2s}){\int_{\mathcal{R}^{n+1}_k\setminus B_r}|y|^a|\nabla u|^2} + (1+Cr^{n+2s}){\int_{B_r}|y|^a|\nabla v|^2}\\&+
\int_{B_r \cap \{y=0\}}\chi_{\{v\neq 0\}}d\mathcal{H}^{n-1} +\int_{(\mathcal{R}^{n}_k\setminus B_r) \cap \{y=0\}}\chi_{\{u\neq 0\}}d\mathcal{H}^{n-1}. \nonumber
\end{align}
Combining the terms of $u$ with the left hand side of \eqref{8.5}, we have
\begin{align}
\int_{B_r}|y|^a|\nabla u|^2 + \int_{B_r \cap \{y=0\}}\chi_{\{u\neq 0\}}d\mathcal{H}^{n-1} \leq  (1&+Cr^{n+2s})\int_{B_r}|y|^a|\nabla v|^2  + \int_{B_r \cap \{y=0\}}\chi_{\{v\neq 0\}}d\mathcal{H}^{n-1} \\&+ Cr^{n+2s} \int_{\mathcal{R}^{n+1}_k\setminus B_r}|y|^a|\nabla u|^2. \nonumber
\end{align}
We note that if the third term can be reabsorbed by the left hand side for some $\alpha_0>0$ in the following way, then the Lemma is proved.
\begin{equation}
r^{n+2s}\int_{\mathcal{R}^{n+1}_k\setminus B_r}|y|^a|\nabla u|^2 \leq Cr^{\alpha_0}\left(\int_{B_r}|y|^a|\nabla u|^2 + \int_{B_r \cap \{y=0\}}\chi_{\{u\neq 0\}}d\mathcal{H}^{n-1}\right). \nonumber
\end{equation}
This follows from the Weiss monotonicity formula \eqref{W(1,u_0)} that there exists $C > 0$, such that for all $r$ small,
\begin{equation}
\int_{B_r}|y|^a|\nabla u|^2 + \int_{B_r \cap \{y=0\}}\chi_{\{u\neq 0\}}d\mathcal{H}^{n-1} > Cr^n. \nonumber
\end{equation}
This finishes the proof of Lemma.
\end{proof}

\begin{thm}
Let $u_k$ be the corresponding eigenfunction to a $\lambda_k$-minimizer. Let $u$ be the extension of $u_k$. Assume that $\|u\|_{C^{0,s}(B_1)} \leq C$, and $|u-U|\leq \tau_0$ in $B_1$ where $U$ (depending only on $x_n,y$), in polar coordinates $x_n = \rho \cos(\theta)$, $y = \rho \sin(\theta)$,
\begin{equation}
U(x,y) = \left(\rho^{1/2}\cos(\theta/2)\right)^{2s}. \nonumber
\end{equation}
If $\tau_0$ and $\kappa$ is small and depend on $\alpha,n,s$, then $F(u) = \D \{u_k \neq 0\}$ is $C^{1,\gamma_0}$ in $B_{1/2}$, where $\gamma_0$ is some number depending on $\alpha,n,s$.
\end{thm}
\begin{proof}
By Theorem \ref{holders}, $u \in C^{0,s}(B_1)$. By previous Lemma, $u$ is an almost minimizer of $\tilde{J}$, and thus Theorem 7.4\cite{AllenGarcia} applies.
\end{proof}

\section{The Global Configuration and a Toy Problem}\label{Section:global}

In light of Theorem \ref{separation}, the $\lambda_k$-minimizer $A \subset \mathcal{R}^{n}_k$ is a disjoint union of connected components for $k>2$. We wish to study how translating the disconnected components of $A$ affect $\lambda_k$. 

Let $u_k$ be the eigenfunction corresponding to the minimizing $\lambda_k$. Let $u = u_k|_{\RR^n_i}$ for any $1\leq i \leq k$, choose $\|u\|_{L^2(\RR^n)} = 1$. Let $A = \{u \neq 0\} \subset \RR^n$, then we can write
\begin{equation}
A = A_1 \cup A_2,\text{ with } A_1 \cap A_2 = \emptyset.
\end{equation}
Let $C = \RR^n \setminus \{A_1 \cup A_2\}$. We investigate how translating $A_1$ and $A_2$ affects $\lambda_k$. We decompose the region of integration $\RR^n \times \RR^n$ in the following way.
\begin{align}\label{8.2}
\lambda_k(A) = \int_{\RR^n}\int_{\RR^n} \frac{(u(x)-u(y))^2}{|x-y|^{n+2s}} = \int_{A_1}\int_{A_1} &+ \int_{A_2}\int_{A_2} + \int_{C}\int_{C} \\&+ 2\int_{A_1}\int_{A_2}+ 2\int_{A_1}\int_{C}+ 2\int_{A_2}\int_{C}.
\end{align}
We first note that the integrals on $A_1\times A_1$, $A_2\times A_2$, and $C\times C$ (which is $0$), are invariant under translation. To study the cross terms, we have
\begin{align}
\int_{A_1}\int_{C} \frac{(u(x)-u(y))^2}{|x-y|^{n+2s}}
&=\int_{A_1}\int_{C} \frac{u(y)^2}{|x-y|^{n+2s}}\\
&=\int_{A_1}\int_{\RR^n} \frac{u(y)^2}{|x-y|^{n+2s}} - \int_{A_1}\int_{A_1} \frac{u(y)^2}{|x-y|^{n+2s}} -\int_{A_1}\int_{A_2} \frac{u(y)^2}{|x-y|^{n+2s}}\\
\int_{A_2}\int_{C} \frac{(u(x)-u(y))^2}{|x-y|^{n+2s}}
&=\int_{A_2}\int_{C} \frac{u(y)^2}{|x-y|^{n+2s}}\\
&=\int_{A_2}\int_{\RR^n} \frac{u(y)^2}{|x-y|^{n+2s}} - \int_{A_2}\int_{A_1} \frac{u(y)^2}{|x-y|^{n+2s}} -\int_{A_2}\int_{A_2} \frac{u(y)^2}{|x-y|^{n+2s}} \nonumber
\end{align}
where the only non-invariant pieces are the integrals on $A_1\times A_2$. Combining these two integrals with the the integral on $A_1 \times A_2$ in \eqref{8.2} gives
\begin{align}\label{transenergy}
\int_{A_1}\int_{A_2}\frac{(u(x)^2-2u(x)u(y) + u(y)^2)}{|x-y|^{n+2s}} &- \int_{A_1}\int_{A_2} \frac{u(y)^2}{|x-y|^{n+2s}}\\&- \int_{A_2}\int_{A_1} \frac{u(y)^2}{|x-y|^{n+2s}} 
= \int_{A_1}\int_{A_2}\frac{-2u(x)u(y)}{|x-y|^{n+2s}}.
\end{align}
This is the only non-invariant term under translation of $A_1$, $A_2$. If $u$ has the same sign on $A_1,A_2$, $u(x)u(y) >0$ then $\lambda_k$ decreases by moving $A_1,A_2$ closer, and if $u$ has different signs, $u(x)u(y) <0$ then $\lambda_k$ decreases by moving $A_1,A_2$ further. 

If $A$ is connected, we expect $A$ to be a ball by symmetric decreasing rearrangements. If $A \cap \RR^n$ has more than one disjoint connected components, by Theorem \ref{separation}, $u$ has a sign on each connected component. We can choose $A_1$ to be any union of the the connected components with the same sign, and $A_2 = A\setminus A_1$. Since $A$ is a $\lambda_k$-minimizer, $\lambda_k$ cannot be reduced by translating any union of connected components. We believe this is false and one can always reduce the energy in this case. Thus, we expect that a $\lambda_k$-minimizer is a union of balls in different copies of $\RR^n$. 

\begin{conj} \label{conj0}
For $k\geq 2$, let $A$ be a $\lambda_k$-minimizer, and $u_k$ the corresponding eigenfunction. If $\lambda_k$ is simple, then $\{u_k|_{\RR_i^n} \neq 0\}$ is a ball for all $i$.
\end{conj}

We observe that since $u$ is not sign-changing on each of the balls, then $u$ must be the first eigenfunction on each ball. Since $\lambda_k$ is a minimizing eigenvalue, $A$ has to consists of $k$ disjoint balls of the same size. In particular, it follows that
\begin{equation}
\lambda_1(A) = \cdots = \lambda_k(A), \nonumber
\end{equation}
which contradicts that $\lambda_k$ is simple. Therefore, Conjecture \ref{conj0} implies the following. 

\begin{corcon} \label{conj}
For $k\geq 2$, let $A$ be a $\lambda_k$-minimizer, then $\lambda_k(A)$ is not simple.
\end{corcon}

\begin{rem}\hfill\\
\vspace{-0.15 in}
\begin{enumerate}
\item Corollary of Conjecture \ref{conj} is the nonlocal analogy of the conjecture made in \cite{henrotbook} which states that in the local case, such minimizing $\lambda_k$ is not simple.
\item Conjecture \ref{conj0} is true if $A$ consists of $m =2,3$ disjoint connected sets. When there are two disjoint connected sets, if $u$ has the same sign on them, $\lambda_k(A)$ decreases by translating them close together, and if $u$ has different signs on each of them, $\lambda_k(A)$ decreases by translating them further away in the direction opposite of each other. If $m =3$, either they all have the same sign, in which case moving them closer together decreases $\lambda_k(A)$, or exactly one of them has a different sign, in which case, moving this one away from the other two decreases the energy. We note that we can always do above translations smoothly when $m=2,3$, and these perturbations keep the order of the eigenvalues. 
\item If $A$ consists of $m>3$ disjoint connected sets, let $A^+ = \{u >0\}$, $A^- = \{u <0\}$. We note that we can put $A^+$, $A^-$ in different copies of $\RR^n$, or equivalently, put $A^+$ and $A^-$ infinitely away from each other, which reduces the energy in \eqref{8.2}. However, it is not clear that this reduced energy corresponds to the $k-th$ eigenvalue. 
\item As $s \rightarrow 1$, our $(-\Delta)^s \phi = -\Delta \phi$ for any $\phi \in C_0^\infty(\RR^n)$. Thus, this nonlocal problem \eqref{problem} offers a new perspective in the study of the minimizer of the $k-th$ eigenvalue of $-\Delta$.
\end{enumerate}
\end{rem}

In \eqref{transenergy}, if we replace $A_1, A_2$ by their enter of mass, we have the following discrete toy problem.

Let $x = (x_1,\cdots,x_d)$, $x_i \in \RR^n$, with signed mass $m(x_i)$, such that $\sum_{i=1}^d |m(x_i)|^2 =1$. Let the energy of the system be
\begin{equation}\label{discreteE}
E(m(x)) = \sum_{i,j=1}^d \frac{-2m(x_i)m(x_j)}{|x_i-x_j|^{n+2s}}.
\end{equation}
We say that the configuration $x_1,\cdots,x_m$ is stationary and stable if for every $v = (v_1,...,v_d), v_i \in \RR^n$
\begin{equation}\label{stablestationary}
\lim_{t\rightarrow 0}\frac{1}{t}(E(x+tv)-E(x)) = 0,\quad \lim_{t\rightarrow 0}\frac{1}{t^2}(E(x+tv)+E(x-tv)-2E(x))\geq 0.
\end{equation}
An associated conjecture to Conjecture \ref{conj0} for the toy problem is the following.
\begin{conj} \label{conj2}
\eqref{discreteE} doesn't have any local minimum for $d \geq 2$. 
\end{conj}
\begin{rem}\hfill\\
\vspace{-0.15 in}
\begin{enumerate}
\item The discrete toy problem is a simplified version of the minimization problem \eqref{problem}, and understanding this discrete problem would help us in understanding the global configuration of the $\lambda_k$-minimizers. 
\item If we formally take $s = -1$, \eqref{discreteE} coincide with the energy in classical electrostatics
\begin{equation}\nonumber
 \sum_{i,j=1}^d \frac{-2m(x_i)m(x_j)}{|x_i-x_j|^{n-2}}.
\end{equation}
Earnshaw's theorem states that there is no stable and stationary configurations. However, the proof doesn't translate to the case $s \in (0,1)$.
\item For any $x_i,x_j \in \RR^n$, let 
\begin{align}
d_{ij} &= \text{dist}\proj[v]{x_j},\\
r_{ij} &= \sqrt{|x_i-x_j|^2 - d{ij}^2}.\nonumber
\end{align}

If $x = (x_1,\cdots,x_d)$ with mass $m(x_1),\cdots, m(x_d)$ achieves a local minimum of $E$ in \eqref{discreteE}, then \eqref{stablestationary} gives the following explicit expression.
\begin{align}
&\sum_{j\neq i}\frac{m(x_i)m(x_j)(n+2s)d_{ij}}{\sqrt{r_{ij}^2+d_{ij}^2}^{n+2s+2}} = 0,\\
&\sum_{j\neq i} \frac{m(x_i)m_(x_j)[r_{ij}^2- d_{ij}^2(n+2s+1)] }{\sqrt{r_{ij}^2+d_{ij}^2}^{n+2s+4}} \geq 0.\nonumber
\end{align}  
\end{enumerate}
\end{rem}

\section*{Appendix}

\subsection{proof of theorem \ref{minissub}.}
For the convenience of the reader, we include the proof of theorem \ref{minissub}, which follows the argument given in \cite{Bucur2012}.
\begin{lem}
For any $u \in \mathfrak{S}_k$, let $A  = \{u \neq 0\}$. Then, for every $k\in \NN$, there exists a constant $c_k(A)$ depending only on $A$ such that for every $j\leq k$ and for every $v\in \mathfrak{S}_k$ such that $B = \{v \neq 0\} \subset A$, we have:
\begin{equation}
\left|\frac{1}{\lambda_j(A)}-\frac{1}{\lambda_j(B)} \right| \leq c_k(A)d_\gamma(A,B). \nonumber
\end{equation}
\end{lem}
\begin{proof}
Fix $k \in \NN$. We consider the linear subspace $V_k \subset L^2(\mathcal{R}^n_k)$ generated by the first $k$ eigenfunctions of the Dirichlet fractional Laplacian on $A$. $V_k$ is finite dimensional subspace of $\fS_k(A)$. We denote by $R_A: L^2(A) \rightarrow L^2(A)$ the resolvent operator ie. $R(A)f = u_{f,A}$ if $(-\Delta)^s u_{A,f} = f$.

We set $P_k: L^2(A) \rightarrow V_{k}$ the $L^2$ projection on $V_k$ and define the finite rank, positive, self adjoint operators:
\begin{equation}
T_k^A = P_k\circ R_A \circ P_k,\  T_k^B = P_k\circ R_B \circ P_k. \nonumber
\end{equation}
Denoting $\mu_j(T_k^A), \mu_j(T_k^B)$ the $j$th eigenvalues of the operators $T_k^A, T_k^B$, we have:
\begin{equation}
\mu_j(T_k^B) \leq \frac{1}{\lambda_j(B)},\ \mu_j(T_k^A) = \frac{1}{\lambda_j(A)}\ \forall j=1,2,...,k. \nonumber
\end{equation}
Indeed, for every $j$, we have:
\begin{equation}
\mu_j(T_k^B) = \mu_j(P_k\circ R_A \circ P_k) = \mu_j(P_k)^2\frac{1}{\lambda_j(B)} \leq \frac{1}{\lambda_j(B)},\nonumber
\end{equation}
since $\mu_j(P_k) \leq \mu_1(P_k) = 1$. In the same way, we have:
\begin{equation}
\mu_j(T_k^A)\leq \frac{1}{\lambda_j(A)}.\nonumber
\end{equation}
Moreover, let $u_j$ be the $j$th eigenfunction of $R_A$ associated with $\lambda_j(A)$, we have that:
\begin{equation}
T_k^A u_j = P_k\circ R_A \circ P_k u_j = \mu_j(A)u_j,\nonumber
\end{equation}
since $P_ku_j = u_j$. Therefore,
\begin{equation}
\mu_j(T_k^A) = \frac{1}{\lambda_j(A)}.\nonumber
\end{equation}
Now, we have for every j:
\begin{align}
0&\leq \frac{1}{\lambda_j(A)}-\frac{1}{\lambda_j(B)}  \leq \mu_j(T_k^A)-\mu_j(T_k^B)\\
&\leq\|T_k^A-T_k^B\|\\
&= \|P_k\circ R_A \circ P_k-P_k\circ R_B \circ P_k\|\\
&= \sup_{|u|_{L^2}\leq 1}((P_k\circ R_A \circ P_k-P_k\circ R_B \circ P_k)u,u)\\
&=\sup_{|u|_{L^2}\leq 1}((R_A-R_B)P_ku,P_ku).\nonumber
\end{align}
Let $u_i$ be the $i$-th eigenfunction, we note that:
\begin{equation}
|P_ku|_{L^\infty} = \sum_{i=1}^k a_iu_i \leq \sum_{i=1}^k |a_i| |u_i|_{L^\infty} \leq C_k(A),\nonumber
\end{equation}
by Proposition \ref{ulinf}. Therefore,
\begin{align}
((R_A-R_B)P_ku,P_ku) &\leq \int_A|R_A(P_ku)-R_B(P_ku)||P_ku|dx \\
&\leq C_k(A)\int_A|R_A(P_ku)-R_B(P_ku)|dx\\
&\leq 2C_k(A)\int_A R_A(|P_ku|)-R_B(|P_ku|)dx\\
&\leq 2C_k(A)^2 \int_A R_A(1) - R_B(1) dx = 2C_k(A)^2 d_\gamma(A,B),\nonumber
\end{align}
where the last inequality used the weak maximum principle.
\end{proof}

\textit{proof of theorem \ref{minissub}}.\\
\begin{proof}
By the previous lemma. there exists a constant $c_k(A)$ depending on $A$ and $k$ such that
\begin{equation}
\left|\frac{1}{\lambda_k(A)}-\frac{1}{\lambda_k(B)} \right| \leq c_k(A)d_\gamma(A,B).\nonumber
\end{equation}
Since $E(B) - E(A) \leq  \frac{1}{2}d_\gamma(A,B)$, for $\delta$ small enough, such that $\delta \leq \frac{4c_k(A)}{\lambda_k(A)}$ and for every $B$ such that $B \subset A$, and $d_\gamma(A,B) \leq \delta$, we have
\begin{equation}
\lambda_k(B) - \lambda_k(A) \geq c'_k(A)(E(B) - E(A))\nonumber
\end{equation}
for a different constant $c'_k(A)$ depending on $\delta, k, A$. By \eqref{problem}, we have $|A| - |B| \leq \lambda_k(B) - \lambda_k(A)$, combining this and the above gives 
$A$ is a shape subsolution with $\Gamma = \frac{1}{c'_k(A)}$.
\end{proof}

\subsection{Proof of theorem \ref{holders}}
\begin{proof}
Let $x_0$ be a free boundary point. Let $h$ be the $s$-harmonic lifting of $u_\alpha$ in a ball $B_r(x_0)$, that is, 
\begin{align}\nonumber
\begin{cases}
(-\Delta)^s h = 0\ & \hbox{in}~B_r= B_r(x_0),\\
h = u_\alpha & \hbox{in}~\mathcal{R}^n_k \setminus B_r(x_0),
\end{cases}
\end{align}
Let $\tilde{h}$ be the extension of $h$. Let $u$ the the extension of $u_\alpha$, then for $\rho < r < 1$,
\begin{align}
\int_{B_\rho(x_0,0)}|y|^a|\nabla {u}|^2 &\leq 2\int_{B_\rho}|y|^a|\nabla ({u}-\tilde{h})|^2 + 2\int_{B_\rho}|y|^a|\nabla \tilde{h}|^2\\
&\leq 2\int_{\RR^n}\int_{\RR^n}\frac{((u_\alpha-h)(x) - (u_\alpha-h)(y))^2}{|x-y|^{n+2s}} + 2\left(\frac{\rho}{r}\right)^{n+1-|a|}\int_{B_r}|y|^a|\nabla \tilde{h}|^2\\
&\leq Cr^n + C\left(\frac{\rho}{r}\right)^{n+1-|a|}\int_{B_r}|y|^a|\nabla \tilde{h}|^2\\
&\leq Cr^n + C\left(\frac{\rho}{r}\right)^{n+1-|a|}\int_{B_r}|y|^a|\nabla u|^2. \nonumber
\end{align}
where we used  Lemma \ref{monotone} for the second inequality to replace $B_\rho$ by $B_r$, and we used Proposition \ref{4.10} for the third inequality to estimate the double integral. Now choose $\delta < 1/2$ with
\begin{equation}
r = \delta^k,\quad \rho =\delta^{k+1}, \quad \mu=\delta^{n},\nonumber
\end{equation}
we obtain
\begin{equation}
\int_{B_{\delta^{k+1}}}|y|^a|\nabla {u}|^2  \leq C\mu^k + C\mu\delta^{1-|a|}\int_{B_{\delta^k}}|y|^a|\nabla u|^2.\nonumber
\end{equation}
Now choose $\delta$ such that $C\delta^{1-|a|} < 1$. By induction, we have
\begin{equation}
\int_{B_{\delta^k}} |x_n|^a|\nabla u|^2 \leq \frac{C^2}{1-C\delta^{1-|a|}}\mu^{k-1}.\nonumber
\end{equation}
Then for all $r< 1/2$, and a different constant,
\begin{equation} \label{4.35}
\int_{B_r(x',0)}|x_n|^a|\nabla u|^2 \leq C r^{n},
\end{equation}
and we can conclude that
\begin{equation}
\int_{B_r(x',0)}|\nabla u| \leq Cr^{n+s}.\nonumber
\end{equation}
Now we can use Morrey's theorem to conclude that 
\begin{equation} \label{4.31}
|u(x',0)-\overline{u_B}| \leq Cr^s
\end{equation}
so that $u$ is $C^{0,s}$ in the thin space $\mathcal{R}_k^n \times \{0\}$. It is left to show that we have the same H\"older growth in the $|x_n|$ direction. For a fixed point $(y',0)$, we consider
\begin{equation}
u_r(x) = \frac{u((y',0)+xr) - u(y',0))}{r^s}.
\end{equation}
By \eqref{4.35}, we have a universal $L^2$ gradient bound in $B^* = B_{1/2}(0,\cdots,0,1)$. Using \eqref{4.31} for $|u_r|$, we have that the average of $|u_r|$ on $B_{3/2}(0)$ is universally bounded. Thus, the average of $|u_r|$ over $B^*$ is universally bounded. By Poincare inequality in $B^*$, we have
\begin{equation}
\|u_r\|_{{W^{1,2}}(B_{1/2}(0,\cdots,0,1))} \leq C.\nonumber
\end{equation}
By first variation, $\div (|x_{n+1}|\nabla|u_r|) = 0$ if $|x_{n+1}| > 0$. Away from the thin space, we use regularity theory for uniformly elliptic equations and conclude that each $u_r$ is continuous in $B^*$ and we have the weak Harnack inequality
\begin{equation}
\|u_r\|_{L^\infty(B_{1/4}(0,\cdots,0,1))} \leq  C.\nonumber
\end{equation}
This shows H\"older growth off the thin space, that is,
\begin{equation}
\frac{|u(y',0)-u(x)|}{|(y',0)-x|^s} \leq C.\nonumber
\end{equation}
Let $x,y \in B_1$. If $|y_{n+1}| \leq |x-y|$, we may use the above to bound 
\begin{equation}
\frac{|u(x)-u(y)|}{|x-y|^s}.\nonumber
\end{equation}
If $|y_{n+1}| > |x-y|$, we then rescale with 
\begin{equation}
u_r = \frac{u((x',0)+xr) - u(x',0))}{r^s} \nonumber
\end{equation}
and use interior gradient bounds on uniformly elliptic equations to conclude 
\begin{equation}
\frac{|u(x)-u(y)|}{|x-y|^s} \leq C.\nonumber
\end{equation}
\end{proof}

\textbf{Data Availability} \\The manuscript has no associated data.\\

\textbf{Statements and Declarations} \\The author has no relevant financial or non-financial interests to disclose.

\nocite{Bucur2012}
\nocite{hitchihiker}
\nocite{nonlocalbook}
\nocite{BucurDalMaso}
\bibliographystyle{plain}
\bibliography{bib}

\end{document}